\documentclass[12pt]{amsart}%
\usepackage{amsfonts}
\usepackage{amsmath}
\usepackage{amssymb}
\usepackage{amsthm}
\usepackage{mathrsfs}
\usepackage{graphicx}%
\usepackage{amsmath}
\usepackage{setspace}
\usepackage{hyperref}
\usepackage{cite}
\usepackage{amsfonts}
\usepackage{algorithmic}
\usepackage{textcomp}
\allowdisplaybreaks[4]
\makeatletter
\@addtoreset{equation}{section}
\makeatother

\newtheorem{theorem}{Theorem}[section]

\newtheorem{corollary}[theorem]{Corollary}

\newtheorem{definition}[theorem]{Definition}

\newtheorem{lemma}[theorem]{Lemma}

\newtheorem{proposition}[theorem]{Proposition}
\newtheorem{remark}[theorem]{Remark}

\makeatletter
\def\@makefnmark{}
\makeatother

\begin{document}

\title[Quantization property of n-Laplacian mean field equation]{Quantization property of n-Laplacian mean field equation and sharp Moser-Onofri inequality}

\author{Lu Chen$^{*}$}
\address[Lu Chen]{Key Laboratory of Algebraic Lie Theory and Analysis of Ministry of Education, School of Mathematics and Statistics, Beijing Institute of Technology, Beijing
100081, PR China}
\email{chenlu5818804@163.com}

\author{Guozhen Lu}
\address[Guozhen Lu]{Department of Mathematics, University of Connecticut, Storrs, CT 06269, USA}
\email{guozhen.lu@uconn.edu}

\author{Bohan Wang}
\address[Bohan Wang]{Key Laboratory of Algebraic Lie Theory and Analysis of Ministry of Education, School of Mathematics and Statistics, Beijing Institute of Technology, Beijing
100081, PR China}
\email{wangbohanbit2022@163.com}
\address{}

\thanks{$*$ Corresponding author.}
\thanks{The first author was partly supported by the National Key Research and Development Program (No.
2022YFA1006900) and National Natural Science Foundation of China (No. 12271027). The second author was partly supported by a Simons grant and a Simons Fellowship from the Simons Foundation.}

\begin{abstract}
 In this paper, we are concerned with the following $n$-Laplacian mean field equation

\[
\left\{ {\begin{array}{*{20}{c}}
   { - \Delta_n  u  = \lambda e^u} & {\rm in} \ \ \Omega,  \\
   {\ \ \ \ u = 0} &\  {\rm on}\  \partial \Omega,
\end{array}} \right.
\]
\[\]
where $\Omega$ is a smooth bounded domain of $\mathbb{R}^n \ (n\geq 2)$ and $- \Delta_n  u =-{\rm div}(|\nabla u|^{n-2}\nabla u)$. We first establish the quantization property of solutions to the above $n$-Laplacian mean field equation. As an application, combining the Pohozaev identity and the capacity estimate, we obtain the sharp constant $C(n)$ of  the Moser-Onofri inequality in the $n$-dimensional unit ball $B^n:=B^n(0,1)$,

\[
    \mathop {\inf }\limits_{u \in W_0^{1,n}(B^n)}\frac{1}{ n C_n}\int_{B^n}  | \nabla u|^n dx- \ln \int_{B^n} {e^u} dx\geq C(n),
\]
which extends the result of  Caglioti-Lions-Marchioro-Pulvirenti in \cite{Caglioti} to the case of $n$-dimensional ball. Here  $C_n=\left(\frac{n^2}{n-1}\right) ^{n-1} \omega_{n-1}$ and $\omega_{n-1}$ is the surface measure of $B^n$.
For the Moser-Onofri inequality in a general bounded domain of $\mathbb{R}^n$, we apply the technique of $n$-harmonic transplantation to give the optimal concentration level of the Moser-Onofri inequality and obtain the criterion for the existence and non-existence of extremals for the Moser-Onofri inequality.
\end{abstract}

\maketitle

\section{Introduction\label{introduction}}

The main content of this paper focuses on the quantization property of the solution of the $n$-Laplacian mean field equation and its application to the    sharp constant of  the Moser-Onofri inequality, as well as the existence and non-existence of extremal functions of the Moser-Onofri inequality. Mean field equations and Moser-Onofri inequalities have significant applications in geometric analysis, harmonic analysis and nonlinear partial differential equations. Let us briefly  present the history of the main results in this direction.
\medskip

Let $\Omega$ be a smooth bounded domain of $\mathbb{R}^n$ ($n\geq 2$) and denote by $W_0^{1,n}(\Omega)$  the closure of $C_c^\infty(\Omega)$ under the Dirichlet norm $\left(\int_\Omega |\nabla u|^ndx\right)^\frac{1}{n}$. The classical Trudinger-Moser inequality (see \cite{Moser1}) states that
\begin{equation}
\mathop {\sup }\limits_{u \in W_0^{1,n}(\Omega),\ \|\nabla u\|_n\leq1} \int_\Omega e^{\alpha_n u^{\frac{n}{n-1}}}dx< +\infty,
\label{TM1}
\end{equation}
where $\alpha_n=n \omega_{n-1}^{\frac{1}{n-1}}$ refers to the sharp constant and $\omega_{n-1}$ denotes the $n-1$ dimensional measure of unit sphere in $\mathbb{R}^n$. The Trudinger-Moser inequality in bounded domain of $\mathbb{R}^n$ has also been extended to bounded domain of Heisenberg group and complex sphere (see \cite{CohnLu1, CohnLu2}). An immediate consequence of the Trudinger-Moser inequality is the following Moser-Onofri inequality (see also \cite{Bec, Ono})

\begin{equation}
   \mathop {\inf }\limits_{u \in W_0^{1,n}(\Omega)}\frac{1}{ n C_n}\int_{\Omega}  | \nabla u|^n dx- \ln \int_{\Omega} {e^u} dx>-\infty,
\label{MO1}
\end{equation}
\[\]
where $C_n=\left(\frac{n^2}{n-1}\right)^{n-1} \omega_{n-1}$.
The critical point of the above inequality (\ref{MO1}) satisfies the following $n$-Laplacian mean field equation

 \begin{equation}
\left\{ {\begin{array}{*{20}{c}}
   { \ \ - \Delta_n  u  = \frac{\rho e^u}{\int_\Omega  e^u dx} } & {\text{in} \ \Omega, } \\
   {u = 0} & {\ \ \text{on} \ \partial \Omega,}
\end{array}} \right.
\label{mfeq1}
\end{equation}
\[\]
where $\rho=C_n$.
\medskip

As $n=2$, the aforementioned equation reduces to the classical mean field equation:

 \begin{equation}
\left\{ {\begin{array}{*{20}{c}}
   { \ \ \ \ - \Delta  u  = \frac{\rho e^u}{\int_\Omega  e^u dx} } & {{\rm in} \ \Omega, }  \\
   {u = 0} & {\ \ {\rm on} \ \partial \Omega,}
\end{array}} \right.
\label{mfeq2}
\end{equation}
\[\]
which arises in the study of Chern-Simons Higgs theory (see \cite{HKP, JW}). For $\rho<8\pi$, the functional related with equation (\ref{mfeq2}) has the compactness and the existence of solutions directly follows from the standard variational method. For $\rho=8\pi$, the existence of solutions is non-trivial due to the loss of  compactness of the related functional. In fact, many authors have found that the existence of solutions depends on the geometry of $\Omega$ in a subtle way. For example, when $\Omega$ is a ball, a consequence of the Pohozaev identity implies the non-existence of solutions for the mean field equation \eqref{mfeq2}; when $\Omega$ is a long and thin domain, the authors of  \cite{Caglioti} proved that the mean field equation admits a positive solution. For $\rho> 8\pi$, the existence of solutions of the mean field equation (\ref{mfeq2}) is a challenging problem.
The construction of Bahri-Coron \cite{Bahri} makes it possible to obtain the existence of mean field solutions on
domains with non-trivial topology. In fact, Ding-Jost-Li-Wang \cite{Ding} established the existence of solutions for $\rho \in(8\pi,16\pi)$ if $\Omega$ is a smooth bounded domain $\Omega\subseteq \mathbb{R}^2$ whose complement contains a bounded region.
Furthermore, they also obtained the similar existence result for the following mean field equation on a closed Riemann surface $(M, g)$ with genus greater than one:

$$-\Delta_g u=\rho\left(\frac{ e^u}{\int_M  e^u dx}-1\right) \ \ {\rm in}\ \ M.$$
\[\]
Struwe-Tarantello \cite{Struwe} proved a similar result for $\rho\in (8\pi, 4\pi^2)$ on the flat torus. For a general closed surface, Malchiodi \cite{Malchiodi} utilized the barycenter technique and proved the existence for $\rho\neq 8\pi \mathbb{N}$.
\medskip

In the study of the existence of solutions for mean field equation (\ref{mfeq2}), an important tool is to establish its quantization property. This  dates back to Brezis and Merle's work in \cite{Brezis}. Lately, many authors, including Nagasaki-Suzuki \cite{Nagasaki}, Li-Shafrir\cite{Li2} and Ma-Wei \cite{Wei3}, etc., have also studied extensively the quantization property of  mean field equation (\ref{mfeq2}). Their results can be stated as follows:
\medskip

\noindent\textbf{Theorem A}:\ \ Let $\{u_{\rho_k}\}$ be a sequence of solutions satisfying the mean field  equation

\[
\left\{ {\begin{array}{*{20}{c}}
   { \ \ \ \ - \Delta  u  = \frac{\rho_k}{\int_\Omega  e^u dx} e^u} & {{\rm in} \ \Omega ,}  \\
   {u = 0} & {\ \ {\rm on} \ \partial \Omega}
\end{array}} \right.
\]
\[\]
with $\rho_k\leq C$.
\medskip

(a) If $\|u_{\rho_k}\|_{L^\infty}$ is bounded, then there exists some function $u\in W_0^{1,2}(\Omega)$ such that $u_{\rho_k}\rightarrow u$ in $C^2(\Omega)$.
\medskip

(b) If $\|u_{\rho_k}\|_{L^\infty}$ is unbounded, then $u_{\rho_k}$ must blow up at some finite points set $S=\{x_1,...,x_m\}\subseteq \Omega$. Furthermore, there holds
$$\rho_k\rightarrow 8m\pi \ \ \ {\rm and}\ \ u_{\rho_k}\rightarrow 8\pi \sum\limits_{i = 1}^m {G( x,{x_i})}\ \ \ {\rm in}\ C_{loc}^2(\Omega\backslash S),$$
where $G(x,y)$ satisfies the equation

\[
\left\{ {\begin{array}{*{20}{c}}
   { - \Delta G(x,y) =\delta_x(y) } & {\rm in}\ \Omega, \\
   {G(x,y)=  0 } & \ \ {\rm on}\  \partial\Omega.  \\
\end{array}} \right.\]
\[\]

However, to our knowledge, quantization analysis for solutions of $n$-Laplacian mean field equation (\ref{mfeq1}) is still unknown. The nonlinearity of $n$-Laplacian operator and the lack of Green's representation formula for $n$-Laplacian equation bring significant challenges to the study of the related problem of the $n$-Laplacian mean field equation. In this paper, we address these difficulties and derive the following result:
\medskip

\begin{theorem}
\label{Th1.1th}
Let $0<C_1\le C_2<\infty$ be two positive constants and  $\alpha_n=n \omega_{n-1}^{\frac{1}{n-1}}$  be the sharp constant in the Moser-Trudinger inequality. Assume that
$u_\lambda$ satisfies the equation

   \begin{equation}
\left\{ {\begin{array}{*{20}{c}}
  \medskip

   { - \Delta_n  u  = \lambda e^u\ \ \ \ {\rm in} \ \Omega ,}  \\
     \medskip

   {\ \ \ \ C_1\leq\int_\Omega  {\lambda e^u dx}\leq C_2,}\\
   {\ \ \ \ \ \ \ \ \ u = 0\ \ \ \ \ \ \ {\rm on} \ \partial \Omega .}
\end{array}} \right.
\label{quan}
\end{equation}
\[\]
 Then we have the following:
\medskip

(a) If $\lambda\rightarrow 0$, the solution $u_\lambda$ must blow up at some finite points set $S=\{x_1,...,x_m\}\subseteq \Omega$ as $\lambda\rightarrow 0$. Furthermore, we have

$$\int_\Omega  {\lambda e^{u_\lambda} dx}\rightarrow \left(\frac{n}{n-1}\alpha_n \right)^{n-1}m$$
and
\begin{equation}
u_\lambda(x)\rightarrow  u_0(x) \ \ \  {\rm in} \ \  C_{loc}^1(\Omega \backslash S),
\label{quantitative formula}
\end{equation}
\[\]
where $u_0(x)$ solves the equation

\begin{equation}\begin{cases}
&- \Delta_n u_0 = \sum\limits_{i = 1}^m \left(\frac{n}{n-1} \alpha_n\right)^{n-1}\delta_{x_i}, \ \ x\in \Omega, \ x_i\in S ,\\
& u_0=0,\ \ x\in \partial\Omega.
\label{nGreen fcn}
\end{cases}\end{equation}
\medskip

(b) If  $u_\lambda$ arises blow-up, then $\lambda\rightarrow 0$.
\end{theorem}
\begin{remark}
The usual proof for the analogy of $(b)$ requires complicated blow-up analysis technique and some quantitative calculations. Our proof is based on comparison theorem for $n$-Laplacian operator, avoiding some of the complicated quantitative estimates.
\end{remark}
\medskip

An immediate consequence of Theorem \ref{Th1.1th} leads to
\medskip

\begin{corollary}
\label{cor1}
 Let $\{u_{\rho_k}\}$ be a sequence of solutions satisfying $n$-Laplacian mean field  equation

   \begin{equation}
\left\{ {\begin{array}{*{20}{c}}
   {- \Delta_n  u  =\frac{\rho_k}{\int_\Omega  e^u dx} e^u\ \ \ \ {\rm in} \ \Omega, }  \\
   {\ \ \ \ \ \ \ \   u = 0\ \ \ \ \ \ \ \ \ \ \ \ \ {\rm on} \ \partial \Omega,}
\end{array}} \right.
\label{mfeq4}
\end{equation}
\[\]
with $\rho_k\leq C$. Then we have the following:
\medskip

\medskip

(a) If $\|u_{\rho_k}\|_{L^\infty}$ is bounded, then there exists $u\in W_0^{1,n}(\Omega)$ such that $u_{\rho_k}\rightarrow u$ in $C^1(\Omega)$.
\medskip

\medskip

(b) If $\|u_{\rho_k}\|_{L^\infty}$ is unbounded, then $u_{\rho_k}$ must blow up at some finite points set $S=\{x_1,...,x_m\}\subseteq \Omega$. Furthermore, we have

$$\rho_k\rightarrow \left(\frac{n}{n-1}\alpha_n\right)^{n-1} m\ \ \ {\rm and} \ \ \ u_{\rho_k}\rightarrow u_0(x) \ {\rm in}\ C_{loc}^1(\Omega\backslash S),$$
\[\]
where $u_0(x)$ denotes the equation

\begin{equation*}\begin{cases}
&- \Delta_n u_0 = \sum\limits_{i = 1}^m \left(\frac{n}{n-1} \alpha_n\right)^{n-1}\delta_{x_i}, \ \ x\in \Omega, \ x_i\in S, \\
& u_0=0,\ \ x\in \partial\Omega.
\end{cases}\end{equation*}
\end{corollary}
\medskip

Another interesting problem related to $n$-Laplacian mean field equation is to consider the existence of extremals and the sharp constant for the Moser-Onofri inequality in a bounded domain $\Omega$ of $\mathbb{R}^n$:

\begin{equation}
\mathop {\inf }\limits_{u \in W_0^{1,n}(\Omega)}\frac{1}{ n C_n}\int_{\Omega}  | \nabla u|^n dx- \ln \int_{\Omega} {e^u} dx\geq C(n).
\label{MO2}
\end{equation}
\medskip

When $\Omega$ is a unit ball $B^n$ of $\mathbb{R}^n$, applying the Pohozaev identity, we can derive the nonexistence of extremals of Moser-Onofri inequality \eqref{MO2} (see Lemma \ref{poho1th} and Proposition \ref{JC1th}). Hence, it is plausible to obtain the sharp constant $C(n)$ of the Moser-Onofri inequality by computing the accurate lower bound of optimal concentration for the Moser-Onofri inequality. Indeed, we obtain

\begin{theorem}
\label{Th1.2th}
There holds that

$$\inf_{u\in W^{1,n}_0(B^n)}\Big(\frac{1}{n C_n}\int_{B^n} {|\nabla u|^n} dx- \ln \int_{B^n}{e^u}dx\Big)=\frac{1}{nC_n}\int_{\mathbb{R}^{n}}e^{\eta_0(y)}\eta_0(y)dy+\frac{n-1}{n}\ln \beta_n- \ln C_n,$$
\[\]
where $\beta_n=n\left(\frac{n^2}{n-1}\right)^{n-1}$ and $\eta_0=\ln \big(\frac{\beta_n}{(1+{\left| x \right|^{\frac{n}{n-1}}})^n}\big)$.
\end{theorem}
\medskip

\begin{remark}
Caglioti-Lions-Marchioro-Pulvirenti in \cite{Caglioti} obtained the sharp constant of the Moser-Onofri inequality in two dimensional disk. However, their method based on ODE does not seem to be applicable to the $n$-Laplacian mean field equation. Furthermore,  the calculation of the optimal concentration level of the Moser-Onofri inequality requires  the Green representation formula in dimension two, which is not attainable for the $n$-Laplacian operator.
 We utilize the capacity estimate to overcome this difficulty and achieve the desired result.
\end{remark}
\medskip

\vskip0.1cm
For a general bounded domain $\Omega$, applying Theorem \ref{Th1.2th} and the technique of $n$-harmonic transplantation developed in \cite{Flucher}, we obtain the optimal concentration level of the Moser-Onofri inequality.

\begin{theorem}\label{thm3}
\label{th3}
Assume that $\Omega$ is a smooth bounded domain in $\mathbb{R}^n$ and $x_0\in \Omega$, then

\begin{equation}\begin{split}
F_{\Omega}^{loc}(x_0)&\triangleq\inf \left\{\lim\limits_{k\rightarrow +\infty}\Big(\frac{1}{n C_n}\int_{\Omega} {|\nabla u_k|^n}dx - \ln \int_{\Omega}{e^{u_k}} dx\Big)\ |\ \lim\limits_{k\rightarrow +\infty}\int_{\Omega}e^{u_k}dx=+\infty,\ \frac{e^{u_k}dx}{\int_{\Omega}e^{u_k}dx}\rightharpoonup \delta_{x_0}\right\}\\
&=\inf_{u\in W^{1,n}_0(B^n)}\Big(\frac{1}{n C_n}\int_{B^n} {|\nabla u|^n}dx - \ln \int_{B^n}{e^u} dx\Big)-n\ln \rho_{\Omega}(x_0),
\end{split}\end{equation}
\[\]
where $\rho_{\Omega}(x_0)$ is the $n$-harmonic radius at $x_0$ (see Definition \ref{4-2def}) in Section 4.
\end{theorem}
\medskip

Define

$$C(n,\Omega)=\inf_{u\in W^{1,n}_0(\Omega)}\Big(\frac{1}{n C_n}\int_{\Omega} {|\nabla u|^n} - \ln \int_{\Omega}{e^u} dx\Big).$$
\[\]
Obviously,

$$C(n,\Omega)\leq \inf_{u\in W^{1,n}_0(B^n)}\Big(\frac{1}{n C_n}\int_{B^n} {|\nabla u|^n} - \ln \int_{B^n}{e^u} dx\Big)-n\sup_{x_0\in \Omega}\ln \rho_{\Omega}(x_0).$$
\[\]
Then we can derive the following criterion for the existence of extremals for the Moser-Onofri inequality on a general bounded domain.
\medskip

\begin{theorem}\label{thm4}

\label{Thcri}
If

$$C(n,\Omega)<\inf_{u\in W^{1,n}_0(B^n)}\Big(\frac{1}{n C_n}\int_{B^n} {|\nabla u|^n} - \ln \int_{B^n}{e^u} dx\Big)-n\sup_{x_0\in \Omega}\ln \rho_{\Omega}(x_0),$$
\[\]
then $C(n,\Omega)$ can be achieved by some function $u\in W^{1,n}_0(\Omega)$. In other words, if $C(n,\Omega)$ is not achieved, then

$$C(n,\Omega)=\inf_{u\in W^{1,n}_0(B^n)}\Big(\frac{1}{n C_n}\int_{B^n} {|\nabla u|^n} - \ln \int_{B^n}{e^u} dx\Big)-n\sup_{x_0\in \Omega}\ln \rho_{\Omega}(x_0).$$
\end{theorem}
\medskip

\begin{remark}
Chang-Chen-Lin \cite{Chang1} have obtained the criterion for the existence of extremals of the Moser-Onofri inequality in two dimensional bounded domain through the conformal map.
\end{remark}

\section{the Proof of Theorem \ref{Th1.1th}}

In this section, we will establish the quantization property for positive solutions of the following $n$-Laplacian mean field equation (\ref{quan}):

  \[
\left\{ {\begin{array}{*{20}{c}}
  \medskip

   { - \Delta_n  u  = \lambda e^u\ \ \ \ \ {\rm in} \ \Omega \subseteq \mathbb{R}^{n},}  \\
   \medskip

   { 0<C_1\leq\int_\Omega  {\lambda e^u dx}\leq C_2,}\\
   { \  u = 0\ \ \ \ \ \  \ \ {\rm on} \ \partial \Omega .}
\end{array}} \right.
\]
\[\]
Namely, we shall provide the proof of Theorem \ref{Th1.1th}. The proof is divided into three steps. In Step 1, we show that the solution $u_\lambda$ of equation (\ref{quan}) must blow up at some finite points set $S=\{x_1,...,x_m\}\subseteq \Omega$ as $\lambda\rightarrow 0$. In Step 2, we further prove that

$$\lim\limits_{\lambda\rightarrow 0}\int_\Omega  {\lambda e^u dx}= \left(\frac{n}{n-1}\alpha_n\right)^{n-1} m$$
and

$$
\lim\limits_{\lambda\rightarrow 0}u_\lambda(x)\rightarrow  u_0(x) \ \ \  {\rm in} \ \  C_{loc}^1(\Omega \backslash S),
$$
\[\]
where $u_0(x)$ satisfies the equation

\[
\begin{cases}
&- \Delta_n u_0 = \sum\limits_{i = 1}^m \left(\frac{n}{n-1} \alpha_n\right)^{n-1}\delta_{x_i}, \ \ x\in \Omega,\\
& u_0=0,\ \ x\in \partial\Omega
\label{Green fcn}
\end{cases}\]
\[\]
for $x_i\in S$.
In Step 3, we explain that $\lambda\rightarrow 0$ is indeed equivalent to $u_\lambda$ blowing up.
\medskip

\medskip

\emph{\bf The proof of  Step 1}: We show that $u_\lambda$ must blow up at some finite points set $S=\{x_1,...,x_m\}\subseteq \Omega$ when $\lambda$ approaches to zero.
\medskip

We first prove that $u_\lambda$ is unbounded when $\lambda$ approaches to zero. We argue this by contradiction. If not, there
exists some constant $C$ such that $\|u_\lambda\|_{L^\infty(\Omega)}\leq C$.
One can easily conclude that

\[\lim\limits_{\lambda\rightarrow 0}\lambda\int_\Omega {e^{u_\lambda}} dx\leq \lim\limits_{\lambda\rightarrow 0}\lambda e^{\|u_\lambda\|_{L^\infty(\Omega)}}|\Omega|= 0,
\]
\[\]
which contradicts with the assumption, $\lambda\int_\Omega {e^{u_\lambda}} dx\geq C_1 >0$ of Theorem \ref{Th1.1th}.
\medskip

\medskip

Define the blow-up set

\begin{equation}
S:=\left\{ \begin{array}{*{20}{c}}
     x\in  \overline{\Omega} : & \begin{array}{l}
u_\lambda \ {\rm is \ the\ solutions\   of\  equation}\ (\ref{quan}),
     \ {\rm there\  exists} \\
x_\lambda\in\Omega\ {\rm such\ that} \ u_\lambda(x_\lambda)\rightarrow\infty\ {\rm as}\ x_\lambda\rightarrow x . \\
    \end{array}  \\
\end{array}\right\}
\label{S}
\end{equation}
\[\]
Then we will prove $S=\{x_1,...,x_m\}\subseteq\Omega$ by defining a new set ``${\Sigma} _{\delta} $" and analyzing the relationship between $S=\{x_1,...,x_m\}$ and ``${\Sigma} _{\delta} $".
\medskip

\medskip

Define $\mu_\lambda:=\lambda e^{u_\lambda}dx$, then $\mu_\lambda(\Omega)=\int_\Omega\lambda e^{u_\lambda}dx\leq C$. Hence, there exists a $\mu_0\in\mathfrak{M}(\Omega)$, the set of all real bounded Borel measures on $\Omega$, such that $\mu_\lambda\rightharpoonup \mu_0$ in the sense of measure.
We  also denote by

\begin{equation}
{\Sigma} _{\delta} : = \left\{ x\ |\ x \in \Omega ,\ \exists\  r=r(x),\ {\rm s.t.}\ \mu_0(B^n(x,r))<(\alpha_n-\delta)^{n-1}\right\} \ \ \ {\rm for\  any} \ \delta>0.
\label{6}
\end{equation}
\[\]
We claim
\begin{lemma}
\label{5th}
If $x_0\in {\Sigma} _{\delta}$ , then $u_\lambda \in L^\infty({B^n}(x_0,r))$ for some $r>0$.
\end{lemma}
\medskip

The proof of Lemma \ref{5th} needs the following lemma.
\medskip

\begin{lemma}
\label{3th}
(see \cite{PEspo1})
If $u\in W_0^{1,n}(\Omega)$ is  the weak solution of

\begin{equation}
\left\{ {\begin{array}{*{20}{c}}
   {-  {\rm div} \ \vec{a}(x,\nabla u)= f(u)} & {\rm{in} \ \Omega ,}  \\
   {\ \ \ \ \ \ \ \ \ \ \ \ u = 0} & {\ \ \rm{on} \ \partial \Omega, }  \\
\end{array}} \right.
\label{3}
\end{equation}
\[\]
where the non-negative function $f(u)\in L^1(\Omega)$ and $\vec{a}(x,\Vec{p})$ is a Caratheodory function satisfying the following two conditions:

 \begin{equation}
 |\vec{a}(x,\Vec{p})|\leq c(a(x)+|p|^{n-1}),\ \ \forall p\in \mathbb{R}^n,\ a.e.\ x\in \Omega,
 \label{cond1}
  \end{equation}
 \begin{equation}
\langle\vec{a}(x,\Vec{p})-\vec{a}(x,\Vec{q}), p-q\rangle\geq d|p-q|^{n},\ \ \forall p,q\in \mathbb{R}^n,\ a.e.\ x\in \Omega,
 \label{cond2}
  \end{equation}
  \[\]
  for some $c,d>0$ and $a(x)\in L^{\frac{n}{n-1}}(\Omega)$.
  Then for any $\delta \in(0,\alpha_n)$, there holds that
 \begin{equation}
 \int_\Omega  {\exp \left\{ {\frac{\left( {\alpha_n - \delta } \right)|u|}{{\left\| f \right\|}^{\frac{1}{n-1}}_{L^1\left( \Omega  \right)}}} \right\}dx} \leq C.
 \label{sucritical ineqn}
\end{equation}
\end{lemma}
\medskip

Now we are in the position to prove that $u_\lambda \in L^\infty(B^n(x_0,r))$ for some $r>0$ when $x_0\in {\Sigma} _{\delta}$.
\medskip

\textbf{The proof of Lemma 2.1}: Set $u_\lambda =u_\lambda^1+u_\lambda^2$, equation (\ref{quan}) can be written as
\begin{equation}
\left\{ {\begin{array}{*{20}{c}}
   { - \Delta_n u_\lambda^1=0} & {{\rm in}\ {B^n}( x_0,{\frac{\varepsilon}{2}}),}  \\
   {\ \ \ \ \ \ \ \  u_\lambda^1 = u_\lambda} & {\ \ {\rm on}\ \partial {B^n}( x_0,{\frac{\varepsilon}{2}}) ,}  \\
\end{array}} \right.
\label{harmoniceqn}
\end{equation}
\[\]
and

\begin{equation}
\left\{ {\begin{array}{*{20}{c}}
   {\  -{\rm div}\ \vec{a}(x,\nabla u_\lambda^2) = \lambda e^{u_\lambda}}& {{\rm in}\ {B^n}( x_0,{\frac{\varepsilon}{2}}),}  \\
   {\ \ \ \ \ \ \ \ \ \ \ \ \ \    u_\lambda^2 = 0} & {\ {\rm on}\ \partial {B^n}( x_0,{\frac{\varepsilon}{2}}).}  \\
\end{array}} \right.
\label{lamda2}
\end{equation}
\[\]
It is easy to check that

\[
 -{\rm div}\ \vec{a}(x,\nabla u_\lambda^2)=-{\rm div}\left(|\nabla u_\lambda|^{n-2}\nabla u_\lambda-|\nabla (u_\lambda-u^2_\lambda)|^{n-2}\nabla (u_\lambda-u^2_\lambda)\right),
\]
\[\]
and we can also find that $\vec{a}(x,\nabla u_\lambda^2)$ satisfies the two conditions \eqref{cond1} and \eqref{cond2},
with $a(x)=|\nabla u_\lambda|^{n-1}\in L^{\frac{n}{n-1}}({B^n}( x_0,{\frac{\varepsilon}{2}}))$.
\medskip

By the definition of ${\Sigma} _{\delta}$, there exists $\varepsilon >0$ such that

\[\mu_0({B^n}( x_0,{\frac{\varepsilon}{2}}))<(\alpha_n-\delta)^{n-1}.\]
\[\]
Then it follows from
$\mu_\lambda=\lambda e^{u_\lambda}dx\rightharpoonup \mu_0$ that
$\int_{{B^n}( x_0,{\frac{\varepsilon}{2}})}\lambda e^{u_\lambda} dx<\alpha_n^{n-1}$.
 Applying this and Lemma \ref{3th} into equation \eqref{lamda2},
we deduce that
$ e^{u^2_\lambda}\in {L^{p_1}}({B^n}( x_0,{\frac{\varepsilon}{2}})) $ for some $p_1>1$.
\medskip

Since $u_\lambda$ satisfies equation

 \[
\left\{ {\begin{array}{*{20}{c}}
  \medskip

   {  -{\rm div}\ \vec{a}(x,\nabla u_\lambda)  = \lambda e^{u_\lambda}\ \ \ \ \ {\rm in} \ \Omega,}  \\
   { \ \ \ \ \ \ \  \ \  u_\lambda = 0\ \ \ \ \ \  \ \ {\rm on} \ \partial \Omega ,}
\end{array}} \right.
\]
\[\]
from the $L^1$-boundedness of $\lambda e^{u_\lambda}$ and Lemma \ref{3th}, we obtain that $ e^{u_\lambda}\in {L^q}(\Omega) $
for some $q>0$.
Combining this and the $L^{p_1}$-boundedness of $e^{u_\lambda^2}$ give
 $u^1_\lambda \in L^{n-1}({B^n}( x_0,{\frac{\varepsilon}{2}}))$.
Since $u_\lambda^1$ satisfies the equation (\ref{harmoniceqn}), using Harnack inequality (\cite{Trudinger}) we derive

\[\|u_\lambda^1\|_{L^\infty({B^n}( x_0,{\frac{\varepsilon}{2}}))}\leq C \|u_\lambda^1\|_{L^{n-1}{B^n}( x_0,{\frac{\varepsilon}{2}}))}
\leq C.\]
\[\]
Thus, $\lambda e^{u_\lambda}= \lambda e^{u^1_\lambda}e^{u^2_\lambda}\in L^{p_1}({B^n}( x_0,{\frac{\varepsilon}{2}}))$ .
By quasilinear elliptic regularity estimate (see \cite{Li}),
we conclude that $u_\lambda$ is uniformly bounded in
${B^n}( x_0,{\frac{\varepsilon}{4}})$.
\medskip

Next, we claim that $\mathop \bigcup\limits_{\delta>0}{\Sigma} _{\delta}^c$ is a finite points set.
\begin{lemma}
\label{7th}
Set
$m:=card(\mathop \bigcup\limits_{\delta>0}{\Sigma} _{\delta}^c)$. Then $m$ is a finite value.
\end{lemma}
\medskip

\begin{proof}
By the definition of ${\Sigma} _{\delta}$, formula (\ref{6}), we easily deduce that

\[\mathop \bigcup\limits_{\delta>0}{\Sigma} _{\delta}^c=\left\{x|x\in  \Omega ,\  \mu_\lambda(x)\geq \alpha_n^{n-1}\right\}.\]
Thus,

\[
\alpha_n^{n-1} m\leq\mu_\lambda(x_1)
+\mu_\lambda(x_2)+...+\mu_\lambda(x_m)\leq\mu_\lambda(\Omega)<+\infty,\]
\[\]
that is,
$m<+\infty$.
\end{proof}

Now, we are prepared to prove that the blow-up set $S$ is equal to $\bigcup\limits_{\delta>0}{\Sigma} _{\delta}^c$, which implies that $u_\lambda$ must blow up at some finite points set $S=\{x_1,...,x_m\}\subseteq \Omega$ when $\lambda\rightarrow 0$. Before proving this, we first state a boundary estimate lemma.
\medskip

\begin{lemma}
\label{partial}
There exists  $\delta>0$  and a constant $C=C(\delta, \Omega)$ such that

$$\|u_\lambda\|_{L^\infty(\Omega_\delta)}\leq C(\delta,\Omega),$$
\[\]
where $ \Omega_\delta:=\left\{x\in\Omega\ |\ dist(x,\partial \Omega)\leq 2\delta\right\}$.
\end{lemma}
\begin{proof}
Using the moving-plane technique  combining with Kelvin transform (see Proposition 2.1 of \cite{Lorca}), one can show that for all $x\in \Omega_\delta$, there exist a measurable set $I_x$ and a positive constant $\gamma=\gamma(\Omega)$ such that
\medskip

 (i) $|I_x|\geq \gamma$,\\

 (ii) $I_x\subseteq\{x\in\Omega: dist(x,\partial \Omega)\geq \delta\}$,\\

 (iii) $u(x)\leq u(\xi)$ for all $\xi \in I_x$.
\[\]

We have already known that $e^{u^\lambda}\in L^{q}(\Omega)$ for some $q>0$ by Lemma \ref{3th}. This leads to $u_\lambda\in L^p(\Omega)$ for any $p>1$. Let $\varphi_1$ be the first eigenfunction of $n$-Laplacian operator with Dirichlet boundary condition, obviously $\varphi_1$ is positive and bounded in $C(\Omega)$. Then for any $x\in \Omega_\delta$, there holds

\[\gamma u_\lambda^p(x) \inf_{I_x
} \varphi_1\leq \int_{I_x}u_\lambda^p \varphi_1dy\leq \int_{\Omega}u_\lambda^p \varphi_1dy\leq \|u_\lambda\|_{L^p}^p \|\varphi_1\|_{L^{\infty}}\lesssim 1. \]
This deduces that there exists a constant $C=C(\delta, \Omega)$ such that $u(x)\leq C(\Omega,\delta)$ for any $x\in \Omega_\delta$ and the proof of Lemma \ref{partial} is completed.
\end{proof}
\medskip

By the above boundary estimate lemma, we immediately deduce that the blow-up set $S$ must be included into $\Omega$. Next, we analyze the relationship of the blow-up set $S$ and $\mathop \bigcup\limits_{\delta>0}{\Sigma} _{\delta}^c$.
\medskip

\begin{lemma}
\label{6th}
\[S=\mathop \bigcup\limits_{\delta>0}{\Sigma} _{\delta}^c.\]
\end{lemma}
\medskip

\begin{proof}
This lemma is equivalent to prove $S\subseteq\mathop \bigcup\limits_{\delta>0}{\Sigma} _{\delta}^c$ and $\mathop \bigcup\limits_{\delta>0}{\Sigma} _{\delta}^c\subseteq S$.
\medskip

We first prove
$S\subseteq\mathop \bigcup\limits_{\delta>0}{\Sigma} _{\delta}^c$. One can argue it by contradiction. If not, there exists $x \in S$ such that $x \in{\Sigma} _{\delta_1 }$ for some $\delta_1>0$. It follows from Lemma \ref{5th} that $u_\lambda \in L^\infty({B^n}(x,r))$
for some $r>0$, which is a contradiction with the definition of $S$.
\medskip

Conversely, for the proof of $\mathop \bigcup\limits_{\delta>0}{\Sigma} _{\delta}^c\subset S$,  we also prove it by contradiction. If not, there exists some $x\in \mathop \bigcup\limits_{\delta>0}{\Sigma} _{\delta}^c$ such that $x\in S^c$, then $u_\lambda \in L^\infty({B^n}(x,r))$ for some $r>0$. Then it follows that
$$\mu(B^n(x,r))=\lim\limits_{\lambda\rightarrow 0}\lambda \int_{B^n(x,r)}e^{u_\lambda}dx=0.$$
This arrives at a contradiction with assumption $x\in \mathop \bigcup\limits_{\delta>0}{\Sigma} _{\delta}^c$. Then we accomplish the proof of Lemma \ref{6th}.
\end{proof}
\medskip

\emph{\bf The proof of Step 2}: We show that as $\lambda\rightarrow 0$,
$u_\lambda(x)\rightarrow  u_0(x)$ in $C_{loc}^1(\Omega \backslash S)$,
where $u_0(x)$ satisfies equation \eqref{nGreen fcn}.
\medskip

From Lemma \ref{6th}, $S=\mathop \bigcup\limits_{\delta>0}{\Sigma} _{\delta}^c$, we get that $\mu_0(x_i)\geq\alpha_n^{n-1}$ for $x_i\in S$, $i=1,2,...,m$. For $x\in \Omega\backslash S$, from Lemma \ref{5th} we know that $u_\lambda$ is $L^\infty$-bounded in ${B^n}(x,r)$, this gives

\[\lim\limits_{\lambda\rightarrow 0} \mu_\lambda({B^n}(x,r))= \lim\limits_{\lambda\rightarrow 0}\int_{{B^n}(x,r)}\lambda e^{u_\lambda}dx= 0.
\]
\[\]
Then it implies that

$$\mu_\lambda\rightharpoonup \mu_0=\sum\limits_{i = 1}^m \mu_0(x_i)\delta_{x_i}.$$
\medskip

Next, we claim that
\begin{lemma}
\label{8th}
$u_\lambda(x)\rightarrow u_0(x)$ in $C_{loc}^1(\Omega\backslash S)$ as $\lambda\rightarrow 0$, where $u_0(x)$ satisfies the equation

\begin{equation}\begin{cases}
&- \Delta_n u_0 = \sum\limits_{i = 1}^m \mu_0(x_i)\delta_{x_i},\ \ x\in \Omega,  \ x_i\in S,\\
& u_0(x)=0,\ \ x\in \partial\Omega.
\end{cases}\end{equation}
\end{lemma}
\medskip

\begin{proof}
Since $u_\lambda\in W^{1,n}_0(\Omega)$ satisfies the equation

 \begin{equation}
\label{2.6-1}
{- \Delta_n u_\lambda = \lambda e^{u_\lambda} \in L^1( \Omega)},
\end{equation}
\[\]
testing equation \eqref{2.6-1} with $u_\lambda^t:=\min\{u_\lambda,t\}$,
we obtain that

\[ \int_{\Omega}  {|\nabla u_\lambda^t|^n dx}
   =\int_{\Omega}\lambda e^{u_\lambda^t}u_\lambda^t dx \leq C_2 t
  \]
  \[\]
  according to the assumption of Theorem \ref{Th1.1th}.
Assume $|\Omega|=|{B^n}(0,r)|$, where ${B^n}(0,r)$ is the ball centered at origin with the radius equal to $r$.
Let $u_\lambda^*$ be the classical rearrangement of $u_\lambda^t$ and $|{B^n}(0,\rho)|=\left|\{x\in {B^n}(0,r):u(x)\geq t\}\right|$. According to the properties of classical rearrangement, we have that $u_\lambda^*\in W^{1,n}_0(B^n(0,r))$ with $u_\lambda^*=t$ in ${B^n}(0,\rho)$. Then it follows that
\begin{equation}
\label{2.6}
\mathop {\inf }\limits_{\phi\in W^{1,n}_0({B^n}(0,r)),\ \phi|_{{B^n}(0,\rho)}=t}\int_{{B^n}(0,r)} {|\nabla \phi|^n}dx  \leq \int_{{B^n}(0,r)} {|\nabla u_\lambda^*|^n}dx \leq C_2 t.
\end{equation}
\[\]
Now, we show the infimum on the left-hand side of \eqref{2.6} is attained by

  \[ \phi _1  ( x)=\left\{ {\begin{array}{*{20}c}
   \medskip

   {t \ln  \frac{r} {|x|}/{\ln \frac{r} {\rho}} \ \ \  {\rm in}\ {B^n}(0,r)\backslash {B^n}(0,\rho)},  \\
   {t\ \ \ \ \ \ \   {\rm in}\ {B^n}(0,\rho).}
 \end{array} } \right.\]
 \[\]

 Consider the infimum
 \begin{equation}\label{2.12}\mathop {\inf }\limits_{\varphi\in W^{1,n}({B^n}(0,r)\setminus {B^n}(0,\rho)),\ \varphi|_{\partial{B^n}(0,\rho)}=t, \varphi|_{\partial\partial{B^n}(0,r)}=0}\int_{{B^n}(0,r)} {|\nabla \varphi|^n}dx,
 \end{equation}
 and we assume that the infimum is attained by some $\varphi$. Define $\phi=\varphi$ in ${B^n}(0,r)\setminus {B^n}(0,\rho)$ and $\phi=t$ in
 ${B^n}(0,\rho)$, the $\varphi$ is obviously the extremal of left-hand side of \eqref{2.6}.
 For the attainability of the infimum \eqref{2.12}, it can be seen as the extremal problem of Dirichlet energy with the Dirichlet boundary condition in the annular domain. Obviously, the infimum is attained since the Dirichlet integral is a convex functional. Hence the extremal function $\varphi$ satisfies the Euler-Lagrange equation
 \begin{equation}\begin{cases}
 &-\Delta_n \varphi=0 \ {\rm in}\ {B^n}(0,r)\setminus {B^n}(0,\rho),\\
 &\varphi=0\ \ {\rm in}\ \partial{B^n}(0,r),\ \ \varphi=t\ \ {\rm in}\ \partial{B^n}(0,\rho).
 \end{cases}\end{equation}
The above equation has a unique solution because of comparison principle. Direct calculation gives that $a\log (|x|)+b$ is $n$-harmonic function in the annular domain ${B^n}(0,r)\setminus {B^n}(0,\rho)$. Through the boundary condition of $\varphi$ in ${B^n}(0,r)\setminus {B^n}(0,\rho)$, we can see that $\varphi=t \ln  \frac{r} {|x|}/{\ln \frac{r} {\rho}}$ in ${B^n}(0,r)\setminus {B^n}(0,\rho)$. This proves that  \[ \phi _1  ( x)=\left\{ {\begin{array}{*{20}c}
   \medskip

   {t \ln  \frac{r} {|x|}/{\ln \frac{r} {\rho}} \ \ \  {\rm in}\ {B^n}(0,r)\backslash {B^n}(0,\rho)},  \\
   {t\ \ \ \ \ \ \   {\rm in}\ {B^n}(0,\rho).}
 \end{array} } \right.\]
 \[\] is the extremal of left-hand side of \eqref{2.6}. Calculating $\|\nabla \phi_1\|_n^n$, by \eqref{2.6}, we get $\rho \leq r e^{-C_2t}$.
 Thus,

$$\left|\{x\in \Omega:u_\lambda\geq t\}\right|=|{B^n}(0,\rho)|\leq \frac{\omega_{n-1}}{n} r^n e^{-C_2}.$$
\[\]
Using Taylor's expansion formula, for any $0<v<n C_2$,

\[\int_{\Omega}  {e^{v u_\lambda} dx}\leq e^v |\Omega|+\sum\limits_{i = 1}^\infty e^{v (m+1)}\left|\{x\in \Omega: m\leq u_\lambda\leq m+1\}\right|
    \leq C(n),\]
\[\]
which implies that $u_\lambda$ is uniformly bounded in $L^n(\Omega)$.
\medskip

Testing equation \eqref{2.6-1} with $\ln \frac{1+2u_\lambda}{1+u_\lambda}$ and applying Young's inequality,
for any $1<q<n$ we deduce that

\begin{align*}
   \int_{\Omega}  {|\nabla u_\lambda|^q dx}&=\int_{\Omega} \frac {|\nabla u_\lambda|^n }{(1+u_\lambda)(1+2u_\lambda)}dx
   +\int_{\Omega}( (1+u_\lambda)(1+2u_\lambda))^\frac{q}{n-q}dx
    \\
    &\leq C(n)\ln 2+C(n)\int_{\Omega}  {e^{v u_\lambda} dx}\leq C(n).
\end{align*}
\[\]
Namely,
one can derive that $u_\lambda$ is uniformly bounded in $W_0^{1,q}(\Omega)$ for any $1<q<n$. Then, there exists $u_0\in W^{1,q}_0(\Omega)$ such that $u_\lambda \rightharpoonup u_0$ in $W^{1,q}_0(\Omega)$, where $u_0$ satisfies

\begin{equation}\begin{cases}
&- \Delta_n u_0 = \sum\limits_{i = 1}^m \mu_0(x_i)\delta_{x_i},\ \ x\in \Omega, \ x_i\in S, \\
& u_0(x)=0,\ \ x\in \partial\Omega.
\end{cases}\end{equation}
\medskip

Due to the definition of $S$, we know that $\lambda e^{u_\lambda}$ is uniformly bounded in $L^\infty_{loc}(\Omega\backslash S)$. Applying the regularity estimate for quasilinear differential operator (see \cite{Li}),
we deduce that $u_\lambda\rightarrow u_0$ in $C_{loc}^1(\Omega\backslash S)$.
\end{proof}
\medskip

Furthermore,  we will present an accurate expression of $\mu_0(x_i)$ in the following lemma.
\begin{lemma}
\label{9th}
For any $i=1,...,m$, $\mu_0(x_i)=\left(\frac{n}{n-1} \alpha_n\right)^{n-1}$.
\end{lemma}
\medskip

To show  Lemma \ref{9th}, we  state the Pohozaev identity for equation (\ref{quan}).
\medskip

\begin{lemma}
\label{poho1th}
For any $i=1,...,m$,

\begin{equation}
\label{Poho}
\begin{split}
    n^2{\int_{{B^n}(x_i,r)}  {F(u)dx} }
    & = \frac{n}{2}\int_{\partial {B^n}(x_i,r)}F(u)\frac{{\partial (|x-x_i|^2)}}{{\partial n}}dS
    -{\int_{\partial{{B^n}(x_i,r)}}\left| {\nabla u} \right|^n(x-x_i,\nu)dS}
    \\
    &\ \ \ +n\int_{\partial {B^n}(x_i,r)}|\nabla u|^{n-2}(\nabla u,\nu)(x-x_i,\nabla u)dS,
\end{split}
\end{equation}
where $F(u)=\int_0^u \lambda e^sds $.
\end{lemma}
\medskip

\begin{proof}
We multiply the equation (\ref{quan}) by $(x-x_i)\cdot\nabla u$ and integrate over ${B^n}(x_i,r)$,

\[{\int_{{B^n}(x_i,r)}  {- \Delta_n u \ \left((x-x_i)\cdot\nabla u\right)dx} }= {\int_{{B^n}(x_i,r)}\lambda e^u \left((x-x_i)\cdot\nabla u\right)dx.}\]
\[\]
We rewrite this expression as

\[A_1=A_2.\]
\[\]
Via the divergence theorem and direct computation, the term on the left is

\begin{align*}
    A_1&={\int_{{B^n}(x_i,r)}  {- \Delta_n u \left ((x-x_i)\cdot\nabla u\right)dx} }
    \\
    &=\frac{1}{n}\int_{\partial {B^n}(x_i,r)}|\nabla u|^n(x-x_i,\nu)dS-\int_{\partial {B^n}(x_i,r)}|\nabla u|^{n-2}(\nabla u,\nu)(x-x_i,\nabla u)dS,
\end{align*}
\[\]
the right-hand side is

 \[A_2=\int_{{B^n}(x_i,r)}\lambda e^u \left((x-x_i)\cdot\nabla u\right)dx=\frac{1}{2}\int_{\partial {B^n}(x_i,r)}F(u)\frac{{\partial (|x-x_i|^2)}}{{\partial n}}dS-n \int_{{B^n}(x_i,r)}F(u)dx.\]
 \[\]
Hence, one can obtain equation (\ref{Poho}).
\end{proof}
\medskip

Then, we turn to prove Lemma \ref{9th}.
\medskip

\begin{proof}
It follows from Lemma \ref{poho1th} that

\begin{align*}
    n^2{\int_{{B^n}(x_i,\varepsilon)}  {F(u_\lambda)dx} }
    &=\frac{n}{2}\int_{\partial {B^n}(x_i,\varepsilon)}F(u_\lambda)\frac{{\partial (|x-x_i|^2)}}{{\partial n}}dS
    -{\int_{\partial{B^n}(x_i,\varepsilon)}\left| {\nabla u_\lambda} \right|^n(x-x_i,\nu)dS}
    \\
    &\ \ \ +n\int_{{B^n}(x_i,\varepsilon)}|\nabla u_\lambda|^{n-2}(\nabla u_\lambda,\nu)(x-x_i,\nabla u_\lambda)dS,
\end{align*}
\[\]
where $F(u_\lambda)={\int_0^{u_\lambda}{\lambda e^sds}}$.
\medskip

Since $u_\lambda$ strongly converges to  $u_0$ in $C^{1}_{loc}(\Omega\backslash S)$ and $u_0$ satisfies equation
 \begin{equation*}\begin{cases}
&- \Delta_n u_0 = \sum\limits_{i = 1}^m \left(\frac{n}{n-1} \alpha_n\right)^{n-1}\delta_{x_i}, \ \ x\in \Omega, \\
& u_0=0,\ \ x\in \partial\Omega,
\end{cases}\end{equation*}
then we deduce that $u_{\lambda}\rightarrow \mu_0^{\frac{1}{n-1}}(x_i)G(x,x_i)+ R(x,x_i)$ in $C_{loc}^1({B^n}(x_i,r)\backslash \{x_i\})$ as
$\lambda\rightarrow 0$, where $G(x,x_i)$ is the Green function of $n$-Laplacian operator with the singularity at $x_i$, $R(x,x_i)$ is continuous at $x_i$ and satisfies $\lim\limits_{x\rightarrow x_i}\left|\nabla(R(x,x_i)-R(x_i, x_i))\right||x-x_i|=0$ (see \cite{Li, Li-Ruf}). Careful calculation gives that

\[\begin{array}{l}
    \medskip

\medskip

   \lim\limits_{\varepsilon\rightarrow 0} \lim\limits_{\lambda\rightarrow 0}\int_{\partial B^n(x_i,\varepsilon)}F(u_\lambda)\frac{\partial (|x-x_i|^2)}{\partial n}dS= 0,
    \\
     \medskip

\medskip

    \lim\limits_{\varepsilon\rightarrow 0} \lim\limits_{\lambda\rightarrow 0}{\int_{\partial{B^n(x_i,\varepsilon)}}\left| {\nabla u_\lambda} \right|^n(x-x_i,\nu)dS} =\omega_{n-1}\left(\frac{n}{\alpha_n}\right)^n\mu_0^{\frac{n}{n-1}}(x_i),
    \\

    \lim\limits_{\varepsilon\rightarrow 0} \lim\limits_{\lambda\rightarrow 0}\int_{\partial B^n(x_i,\varepsilon)}|\nabla u_\lambda|^{n-2}(\nabla u_\lambda,\nu)(x-x_i,\nabla u_\lambda)dS=\omega_{n-1}\left(\frac{n}{\alpha_n}\right)^n\mu_0^{\frac{n}{n-1}}(x_i).
\end{array}\]
\[\]
This together with Pohozaev identity \eqref{poho1th} yields that
\[\lim\limits_{\varepsilon\rightarrow 0} \lim\limits_{\lambda\rightarrow 0} n^2{\int_{{B^n}(x_i,\varepsilon)}  {F(u_\lambda)dx} }= (n-1)\omega_{n-1}\left(\frac{n}{\alpha_n}\right)^n\mu_0^{\frac{n}{n-1}}(x_i).\]
\[\]
On the other hand,
\[ \lim\limits_{\varepsilon\rightarrow 0} \lim\limits_{\lambda\rightarrow 0} n^2\int_{{B^n}(x_i,\varepsilon)}  {F(u_\lambda)dx}= \lim\limits_{\varepsilon\rightarrow 0} \lim\limits_{\lambda\rightarrow 0} n^2 \int_{{B^n}(x_i,\varepsilon)}  {\lambda e^{u_\lambda}dx}=n^2\mu_{0}(x_i). \]
\[\]
Combining the above estimate, we conclude that
\[\mu_0^{\frac{1}{n-1}}(x_i)=\frac{n}{n-1} \alpha_n.\]
\end{proof}
\medskip

\emph{\bf The proof of Step 3}:
Recalling from Step 1, we have proven that $u_\lambda$ must blow up as $\lambda\rightarrow 0$. Hence, to accomplish the proof of Step 3, we only need to prove that if $u_\lambda$ blows up, then $\lambda$ must approach to zero. Since $\lambda\int_{\Omega}e^{u_\lambda}dx$ is bounded, we only need to prove that  $ \lim\limits_{\lambda\rightarrow 0} \int_{\Omega}e^{u_\lambda}dx =+\infty$. By boundary estimate Lemma \ref{partial},
we know that $u_\lambda$ does not blow up at the boundary. If $u_\lambda$ blows up at some point $x_1$, we claim that

\[\mu_0(x_1)\geq\alpha_n^{n-1}.
\]
\[\]
Indeed, suppose not, there exists $\delta>0$ such that $\mu_0(x_1)<( \alpha_n-\delta)^{n-1}$. According to Lemma \ref{5th}, $u_\lambda(x_1)$ is bounded, which is a contradiction.
Using the comparison principle for $n$-Laplacian operator, we get

\[u_\lambda\geq  n\ln \frac{1}{|x-x_1|}.\]
Naturally, we have

\[\lim\limits_{\lambda\rightarrow 0}\int_{{B^n}(x_1,\delta)}e^{u_\lambda }dx\geq\lim\limits_{\lambda\rightarrow 0}\int_{{B^n}(x_1,\delta)}\frac{1}{|x-x_1|^n}dx= +\infty.\]
\[\]
The proof of Theorem $\ref{Th1.1th}$ is completed.
\medskip

\section{the Proof of Theorem \ref{Th1.2th}}
\medskip

In this section, we will show the non-existence of extremal functions for the Moser-Onofri inequality in the ball of $\mathbb{R}^n$ and obtain the accurate value of infimum of the Moser-Onofri inequality in the ball, namely, we shall give the proof of Theorem \ref{Th1.2th}.
\medskip

\medskip

We first show  that the Moser-Onofri inequality in the ball of $\mathbb{R}^n$ does not have an extremal, i.e.,
\medskip

\begin{lemma}

\label{lem1}
$$\inf_{u\in W^{1,n}_0(B^n)}\Big(\frac{1}{n C_n}\int_{B^n} {|\nabla u|^n} dx- \ln \int_{B^n}{e^u}dx\Big)$$
cannot be  achieved in $W_0^{1,n}(B^n)$.
\end{lemma}

\medskip

\begin{proof}
We argue this by contradiction. Indeed, if the infimum

$$\mathop {\inf }\limits_{u\in W^{1,n}_0(B^n)}\Big(\frac{1}{n C_n}\int_{B^n} {|\nabla u|^n}dx - \ln \int_{B^n}{e^u} dx\Big)$$
\[\]
 were achieved, then the extremal function $u$ would satisfy the equation

\begin{equation}
\begin{cases}\label{ad1}
-\Delta_n u=C_n \frac{e^u}{\int_{B^n}{e^u}dx}\ \ \ {\rm in}\ B^n,   \\
\ \ \ \ \ \ \ u=0 \ \ \ \ \ \ \ \ \ \ \ \ \ \ {\rm on}\ \partial B^n.
\end{cases}
\end{equation}
\[\]
Applying the Pohozaev identity to equation (\ref{ad1}), we get

\begin{equation}
\label{Poho2}
    \int_{\partial{B^n}}| \nabla u|^n(x,\nu)dS
    = \frac{n^2}{n-1}\frac{C_n}{\int_{ B^n}{e^u}dx}\int_{ B^n}(e^u-1)dx.
\end{equation}
\[\]
Since $|\frac{\partial u}{\partial n} |=|\nabla u |$ on $\partial{B^n}$ and using H\"{o}lder's inequality, one can calculate that

\[
C_n= \int_{\partial{B^n}}|\nabla u |^{n-1}dS  \leq |\partial B^n|^{\frac{1}{n}} \left(\int_{\partial B^n}|\frac{\partial u}{\partial n}|^n(x,\nu)dS \right)^{\frac{n-1}{n}}<|\partial B^n|^{\frac{1}{n}}\left(\frac{n^2}{n-1}C_n\right)^{\frac{n-1}{n}},
\]
\[\]
which arrives at a contradiction with $C_n=\left(\frac{n^2}{n-1}\right)^{n-1} \omega_{n-1}$. This proves that the Moser-Onofri inequality in $B^n$ actually does not admit any extremal.
\end{proof}

\medskip

\medskip

Next, we start to calculate the accurate value of infimum of the Moser-Onofri inequality

$$\mathop {\inf }\limits_{u\in W^{1,n}_0(B^n)}\frac{1}{n C_n}\int_{B^n} {|\nabla u|^n} dx- \ln \int_{B^n}{e^u} dx.$$
\[\]
The proof can be divided into two parts. In Part 1, we will adopt the method of subcritical approximation and capacity estimate to obtain the lower-bound of the infimum of the Moser-Onofri inequality on the ball. In Part 2, we will construct a suitable test function sequence to show that the lower-bound obtained in Part 1 is actually the infimum of the Moser-Onofri inequality on the ball $B^n$.
\medskip

\medskip

\emph{\bf Part 1}: We start the proof of the lower-bound of the infimum of the Moser-Onofri inequality on the ball. For this purpose, we first show that the subcritical Moser-Onofri inequality

\begin{equation}
\mathop {\inf }\limits_{u \in W_0^{1,n}(B^n)}\frac{1}{n \rho}\int_{B^n}   | \nabla u|^n dx- \ln \int_{B^n} {e^u} dx> -\infty\ \ \ \ where\, \rho<C_n
\
\label{JCN1}
\end{equation}
\[\]
admits an extremal.
\medskip

\begin{lemma}
\label{sub}
Denote

$$J_\rho(u)=\frac{1}{n \rho}\int_{B^n}  | \nabla u|^ndx - \ln \int_{B^n} {e^u} dx.$$
\[\]
Then for $\rho<C_n$,
$\mathop {\inf }\limits_{u \in W_0^{1,n}(B^n)}J_\rho(u)$ can be achieved by some function $u_\rho \in W_{0}^{1,n}(B^n)$.

\end{lemma}
\medskip

\begin{proof}
Let $\left\{ {u_j } \right\} \in W_{0}^{1,n} (B^n)$ be a minimizing sequence for $J_\rho(u)$, i.e.,

\[\mathop {\lim }\limits_{j \to \infty }\frac{1}{n\rho}\int_{B^n}  | \nabla u_j|^n dx- \ln \int_{B^n} e^{u_j} dx
 = J_\rho(u).\]
 \[\]
On the other hand, using the Moser-Onofri inequality \eqref{MO1}, we derive that
$$  \lim\limits_{j\rightarrow +\infty}\frac{1}{ n C_n}\int_{\Omega}  | \nabla u_j|^n dx- \ln \int_{\Omega} {e^{u_j}} dx>-C.$$
Combining the above estimates, we obtain that $u_j$ is bounded in $W_{0}^{1,n} (B^n)$,
which implies that
$e^{u_j} dx  \to e^{u_\rho} dx$ in $L^p (B^n)$ for any $p>1$, where $u_\rho$ is the weak limit of $u_j$ in $W^{1,n}_0(B^n)$.
Then the proof for existence of extremals of $\mathop {\inf }\limits_{u \in W_0^{1,n}(B^n)}J_\rho(u)$ for $\rho<C_n$ is accomplished.
\end{proof}
\medskip

Obviously, $u_\rho$ satisfies

\begin{equation}
\left\{ {\begin{array}{*{20}{c}}
  {  - \Delta_n  u_\rho  = \rho \frac{e^{u_\rho}}{\int_{B^n}{e^{u_\rho}}dx}} & {{\rm in} \ B^n, }  \\
   {u_\rho = 0} & {\ \ {\rm on} \ \partial B^n}.  \\
\end{array}} \right.
\label{eq3}
\end{equation}
\[\]
By the maximum principle and moving-plane method, we know that $u_\rho$ is a radial decreasing function.
Since $u_\rho$ is the extremal function of the subcritical Moser-Onofri inequality (\ref{JCN1}),

\[\frac{1}{n\rho}\int_{B^n}  | \nabla u_\rho|^n dx- \ln \int_{B^n} e^{u_\rho} dx
\leq \frac{1}{n \rho}\int_{B^n}  | \nabla u|^ndx - \ln \int_{B^n} e^u dx,\ \ \forall u\in W_0^{1,n}(B^n).\]
\[\]
 Letting $\rho\rightarrow C_n$, then taking the infimum of both sides of the above inequality, it deduces that

 \[\mathop {\lim }\limits_{\rho \to C_n }\frac{1}{n\rho}\int_{B^n}  | \nabla u_\rho|^n dx- \ln \int_{B^n} e^{u_\rho} dx
\leq \mathop {\inf }\limits_{u \in W_0^{1,n}(B^n )}\frac{1}{n C_n}\int_{B^n}  | \nabla u|^ndx - \ln \int_{B^n} e^u dx.\]
 \[\]
 Using the definition of infimum, it is obvious that

 \[\mathop {\lim }\limits_{\rho \to C_n }\frac{1}{n\rho}\int_{B^n}  | \nabla u_\rho|^n dx- \ln \int_{B^n} e^{u_\rho} dx
 \geq \mathop {\inf }\limits_{u \in W_0^{1,n}(B^n )}\frac{1}{n C_n}\int_{B^n}  | \nabla u|^ndx - \ln \int_{B^n} e^u dx.\]
 \[\]
Then we obtain

\[\mathop {\lim }\limits_{\rho \to C_n }\frac{1}{n\rho}\int_{B^n}  | \nabla u_\rho|^n dx- \ln \int_{B^n} e^{u_\rho} dx
 = \mathop {\inf }\limits_{u \in W_0^{1,n}(B^n)}\frac{1}{n C_n}\int_{B^n }  | \nabla u|^ndx - \ln \int_{B^n} e^u dx.\]
 \[\]
\medskip

Hence, to obtain the infimum of critical Moser-Onofri inequality, we only need to calculate the limit

\[\mathop {\lim }\limits_{\rho \to C_n }\frac{1}{n\rho}\int_{B^n}  | \nabla u_\rho|^n dx- \ln \int_{B^n} e^{u_\rho} dx.
 \]
\medskip

Assume $c_\rho:= \mathop {\max}\limits_{B^n} u_{\rho}(x)$.
 We claim that $c_\rho$ is unbounded and argue this by contradiction. In fact, if $c_\rho$ is bounded, then it follows from the regularity estimate for $n$-Laplacian operator that there exists some $u\in W_{0}^{1,n}(B^n)$ such that $u_\rho \rightarrow u$ in $C^{1}(B^n)$ and $u$ satisfies $n$-Laplacian mean field equation

\begin{equation}
\left\{ {\begin{array}{*{20}{c}}
  { \ \ - \Delta_n  u_\rho  = C_n \frac{e^{u_\rho}}{\int_{B^n}{e^{u_\rho}}dx}} & {{\rm in} \ B^n, }  \\
   {u_\rho = 0} & {\ \ {\rm on} \ \partial B^n },  \\
\end{array}} \right.
\end{equation}
\[\]
which is a contradiction with Lemma \ref{lem1}. Thus, $c_\rho$ is unbounded.
Furthermore, using Corollary \ref{cor1}, one can deduce that
$ u_\rho(x)\rightarrow u_0(x)$ in $C^{1}_{loc}(B^n\backslash \{0\})$,
where $u_0(x)$ satisfies the equation

\begin{equation*}\begin{cases}
&- \Delta_n u_0 = \sum\limits_{i = 1}^m \left(\frac{n}{n-1} \alpha_n\right)^{n-1}\delta_0, \ \ x\in B^n, \\
& u_0=0,\ \ x\in \partial B^n.
\end{cases}\end{equation*}
\[\]
This characterizes the asymptotic behavior of $u_\rho$ away from the blow-up point $0$. Now we start to study the asymptotic behavior of $u_\rho$ around the origin.
Set

\[
\eta_\rho(x):=u_\rho(\varepsilon_\rho x) - {c_\rho}+\ln \beta_n,\ \ \ x\in B^n(0,\varepsilon^{-1}_\rho)\]
\[\]
and

\begin{equation}
\frac{\rho}{\int_{B^n}{e^{u_\rho}}dx} \varepsilon_\rho^n e^{c_\rho-\ln \beta_n} =1.
\label{eta}
\end{equation}
\[\]
Careful computation gives the following equation

\[- \Delta_n \eta_\rho= e^{\eta_\rho} \ \ \ {\rm in} \ \ \ B^n(0,\varepsilon^{-1}_\rho).\]
\[\]
Obviously, $\varepsilon_\rho\rightarrow 0$ as $\rho\rightarrow C_n$. Indeed, if $\varepsilon_\rho\nrightarrow0$, then $\rho\frac{e^{c_\rho}}{\int_{B^n} e^{u_\rho} dx}<+\infty$. Hence $f_\rho=\rho\frac{e^{u_\rho}}{\int_{B^n} e^{u_\rho} dx}$ is $L^{\infty}$ bounded in $B^n$ . Applying the quasilinear estimate into equation $-\Delta_n u_\rho=f_\rho \in L^{\infty}(B^n), \ u_\rho=0\ {\rm in }\ \partial B^n$ (see Theorem 2.2 of \cite{Li}), we conclude that $u_\rho$ is uniformly bounded in $B^n$, which is a contradiction.
\medskip

Since $\eta_\rho\leq 0$ and $e^{\eta_\rho}\in  L^\infty(B^n)$, according to Harnack inequality \cite{Trudinger}, we know that
$\eta_\rho$ is uniformly bounded near origin.
Using quasilinear elliptic estimate again, one can derive that there exists $\eta_0\in C^{1,\alpha}(\mathbb{R}^n)$ such that
$\eta_\rho \rightarrow\eta_0$ in $C_{loc}^{1,\alpha}(\mathbb{R}^n)$.
Then, it follows that

\begin{equation}\begin{split}
\int_{\mathbb{R}^n}e^{\eta_0} dx&=\mathop {\lim }\limits_{R \to +\infty }\int_{B^n (0,R)}   e^{\eta_0} dx
=\mathop {\lim }\limits_{R \to +\infty }\mathop {\lim }\limits_{\rho \to C_n}\int_{B^n (0,R)}   e^{\eta_\rho} dx\\
&\leq \mathop {\lim }\limits_{R \to +\infty }\mathop {\lim }\limits_{\rho \to C_n}\int_{B^n(0,\varepsilon^{-1}_\rho)}   e^{\eta_\rho} dx= C_n.
\end{split}\end{equation}
\[\]
By the classification of solution for Liouville equation and $\eta_0(0)=\ln \beta_n$, we have

 $$\eta_0=\ln\frac{ \beta_n}{ \big(1+| x |^{\frac{n}{n-1}}\big)^n}.$$
\[\]
In summary, we have obtained the asymptotic behavior of $u_\rho$ near and away from origin.
Next, we aim to establish the asymptotic behavior of $\eta_\rho$ at infinity.
Denote $\varphi:=-c_1\ln |x|+\ln \beta_n$ and

\begin{equation}
\left\{ {\begin{array}{*{20}{c}}
  { \varphi|_{\partial B^n(0, R)} : = -c_1\ln R+\ln\beta _n,}  \\
   {\varphi|_{\partial B^n(0, \varepsilon_\rho^{-1})}:=c_1\ln\varepsilon_\rho+\ln \beta_n, }  \\
\end{array}} \right.
\end{equation}
\[\]
where $c_1$ is a undetermined positive constant. If $c_1$ satisfies the following conditions

\begin{equation}
\left\{ {\begin{array}{*{20}{c}}
  { \varphi|_{\partial B^n(0, R)} < \eta_\rho|_{\partial B^n(0, R)},}  \\
   {\varphi|_{\partial B^n(0, \varepsilon_\rho^{-1})}<
   \eta_\rho|_{\partial B^n(0, \varepsilon_\rho^{-1})}. }  \\
\end{array}} \right.
\end{equation}
\[\]
Since $\eta_\rho|_{\partial B^n(0, \varepsilon_\rho^{-1})}=-c_\rho+\ln \beta_n$ and $\mathop {\lim }\limits_{\rho \to C_n }\eta_\rho|_{\partial B^n(0, R)}=\eta_0|_{\partial B^n(0, R)}=\ln\frac{\beta _n}{ ( 1+R^{\frac{n}{n-1}})^n}$, by direct computation we can choose
 $c_1=\frac{n^2}{n-1}+\sigma$ such that

 $$c_1\ln R-n\ln(1+R^{\frac{n}{n-1}})>0.$$
\[\]
Then using the comparison principle, there holds that

$$\eta_\rho\geq -(\frac{n^2}{n-1}+\sigma)\ln |x|+\ln\beta _n\ \ \ {\rm in}\ \ B^n(0, \varepsilon_\rho^{-1})\backslash B^n(0, R).$$
\[\]
Furthermore,
we will show the accurate asymptotic behavior of $\eta_\rho$ at infinity. For simplicity, we only provide an outline of the proof. Let us first recall the Kelvin transform $\hat{\eta}_\rho(x)=\eta_\rho(\frac{x}{|x|^2})$ of
${\eta_\rho}$ satisfies

\begin{equation}
\left\{ {\begin{array}{*{20}{c}}
\vspace{1.5ex}
   { - \Delta_n \hat{\eta}_\rho = \frac{e^{ \hat{\eta}_\rho }}{|x|^{2n}}\ \ \ {\rm in}\ \mathbb{R}^n\backslash B^n(0,\varepsilon_\rho), }& {}  \\
   {\int_{\mathbb{R}^n\backslash B^n(0,\varepsilon_\rho)} \frac{e^{ \hat{\eta}_\rho }}{|x|^{2n}} dx=\rho.\ \ \ \ \ } & {}  \\
\end{array}} \right.
\label{2.24}
\end{equation}
\[\]
Obviously, $\hat{\eta}_\rho\in C^{1,\alpha}(\mathbb{R}^n\backslash B^n(0,\varepsilon_\rho))$.
\medskip

\medskip

Step 1. Decomposing $\hat{\eta}_\rho=\eta_\rho^\epsilon+H_\epsilon$, we fix small $r>0$, and for $0<\epsilon<r$, $H_\epsilon$ satisfies

\begin{equation}
\left\{ {\begin{array}{*{20}{c}}
\vspace{1.5ex}
   { - \Delta_n H_\epsilon = 0\ \ \ {\rm in}\ B^n(0,r)\backslash B^n(0,\epsilon), }& {}  \\
   {H_\epsilon=\hat{\eta}_\rho\ \ \ {\rm on} \ \partial \left (B^n(0,r)\backslash B^n(0,\epsilon)\right).} & {}  \\
\end{array}} \right.
\label{2.26}
\end{equation}
\[\]
Local  H\"{o}lder estimates about equation \eqref{2.26} can be found in \cite{DiBenedetto, Serrin}, we can get that

 $$  H_\epsilon\in C^{1,\alpha}(\overline{B^n(0,r)\backslash B^n(0,\epsilon)}).$$
 \[\]
  Then $\eta_\rho^\epsilon=\hat{\eta}_\rho-H_\epsilon\in C^{1,\alpha}(\overline{B^n(0,r)\backslash B^n(0,\epsilon)})$
satisfies that

\begin{equation}
\left\{ {\begin{array}{*{20}{c}}
\vspace{1.5ex}
   { - \Delta_n (\hat{\eta}_\rho-\eta_\rho^\epsilon) = 0\ \ \ {\rm in}\ B^n(0,r)\backslash B^n(0,\epsilon), }& {}  \\
   {\eta_\rho^\epsilon=0 \ \ \ {\rm on} \ \partial (B^n(0,r)\backslash B^n(0,\epsilon)).} & {}  \\
\end{array}} \right.
\label{2.27}
\end{equation}
\medskip

\medskip

Step 2. By Lemma \ref{3th} and Sobolev embedding Theorem, we find that

$$\|H_\epsilon\|_{L^n( B^n(0,r)\backslash B^n(0,\epsilon))}\leq C.$$
\[\]
Applying Ascoli-Arzela's Theorem and the description in \cite{DiBenedetto, Serrin}, as $\epsilon\rightarrow 0$, there holds that $H_\epsilon\rightarrow H_0$ in $C_{loc}^1(\overline{B^n(0,r)}\backslash \{0\})$, where $H_0$ satisfies that

\begin{equation}
\left\{ {\begin{array}{*{20}{c}}
\vspace{1.5ex}
   { - \Delta_n H_0 = \delta_0\ \ \ {\rm in}\ B^n(0,r), }& {}  \\
   {H_0=\eta_\rho^\epsilon \ \ \ {\rm on} \ \partial B^n(0,r),} & {}  \\
\end{array}} \right.
\end{equation}
\[\]
and
$H_0(x)+\left(\frac{\rho}{\omega_{n-1}}\right)^{\frac{1}{n-1}}\ln |x|\in L^\infty(B^n(0,r))$.
\medskip

Step 3. By comparison principle, as $\epsilon\rightarrow 0$, we have that $\eta_\rho^\epsilon\rightarrow \eta_\rho^0:=\hat{\eta}_\rho-H_0$ in
$C_{loc}^1(\overline{B^n(0,r)}\backslash \{0\})$, where $e^{\eta_\rho^0}\in L^p(B^n(0,r))$ for all $p\geq 1$.
Using Lemma \ref{3th} and Sobolev embedding Theorem again,
as $\epsilon\rightarrow 0$,

$$\|\eta_\rho^0\|_{L^\infty(B^n(0,r)\backslash \{0\})}\leq C.$$
\[\]
Combining Step 1-3, one can easily derive that

\begin{equation}
\label{3-0}
\eta_\rho(x)+\left(\frac{\rho}{\omega_{n-1}}\right)^{\frac{1}{n-1}}\ln |x|\in L_{loc}^\infty(B^n(0,\varepsilon_\rho^{-1})).
\end{equation}
\medskip

Now, we are in the position to use capacity estimate to calculate the value of $\mathop {\lim }\limits_{\rho \to C_n }J_{\rho}(u_\rho)$, and
$$\mathop {\lim }\limits_{\rho \to C_n }J_{\rho}(u_\rho)=\mathop {\lim }\limits_{\rho \to C_n }\frac{1}{n\rho}\int_{B^n}  | \nabla u_\rho|^n dx- \ln \int_{B^n} e^{u_\rho} dx.$$
\medskip

\medskip

\begin{proposition}
\label{JC1th}

\begin{equation}
\mathop {\lim }\limits_{\rho \to C_n }J_{\rho}(u_\rho)
    \geq\frac{1}{nC_n}\int_{\mathbb{R}^{n}}e^{\eta_0(y)}\eta_0(y)dy+\frac{n-1}{n}\ln \beta_n- \ln C_n.
\label{lowerbd}
\end{equation}
\end{proposition}
\medskip

\begin{proof}
In fact, by  equation \eqref{eq3} for $u_\rho$ and equality \eqref{eta}, we infer to

\begin{equation}
   \int_{B^n}| \nabla  u_\rho |^ndx   =\int_{B^n}  {\rho \frac{e^{u_\rho}u_\rho}{\int_{B^n}  e^{u_\rho} dy}} dx=( c_\rho-\ln\beta_n)
   \int_{B^n(0,\varepsilon^{-1}_\rho)}  e^{\eta_\rho(y)}dy
   +\int_{B^n(0,\varepsilon^{-1}_\rho)}  e^{\eta_\rho(y)}\eta_\rho(y)dy
   \label{eq8}
\end{equation}
\[\]
and
\begin{align*}
    \int_{B^n} e^ { u_\rho }dx&=
   \rho\varepsilon_\rho^n e^ { c_\rho -\ln \beta_n}.
\end{align*}
\[\]
Hence,  we calculate directly that

\begin{align*}
\mathop {\lim }\limits_{\rho \to C_n }J_\rho(u_\rho)&=
\mathop {\lim }\limits_{\rho \to C_n }\left(\frac{1}{\rho n}\int_{B^n} | \nabla  u_\rho |^n dx- \ln \int_{B^n} e^ {u_\rho }dx\right)
  \\
  &=-\frac{n-1}{n}
\mathop {\lim }\limits_{\rho \to C_n }\left( c_\rho+\frac{n^2}{n-1}\ln \varepsilon_\rho+n \ln \big( \frac{R^{\frac{n}{n-1}}}{1+R^{\frac{n}{n-1}}}\big)+\int_{B^n(0,\varepsilon^{-1}_\rho)\backslash B^n(0,R)}e^{\eta_\rho}\eta_\rho dy\right)\\
  &\ \ +\frac{1}{n C_n}\mathop {\lim }\limits_{\rho \to C_n }\int_{B^n(0,\varepsilon^{-1}_\rho)}e^{\eta_\rho(y)}\eta_\rho(y)dy+
  \frac{n-1}{n}\mathop {\lim }\limits_{\rho \to C_n }\int_{B^n(0,\varepsilon^{-1}_\rho)\backslash B^n(0,R)}e^{\eta_\rho}\eta_\rho dy\\
  &\ \ +(n-1) \ln \big( \frac{R^{\frac{n}{n-1}}}{1+R^{\frac{n}{n-1}}}\big)+\frac{n-1}{n}\ln \beta_n- \ln C_n.
\end{align*}
\[\]
We will first claim that, for any sufficiently large $R$, there holds that

\[\mathop {\lim }\limits_{\rho \to C_n }\left( c_\rho+\frac{n^2}{n-1}\ln \varepsilon_\rho
+n \ln \big( \frac{R^{\frac{n}{n-1}}}{1+R^{\frac{n}{n-1}}}\big)
+\int_{B^n(0,\varepsilon^{-1}_\rho)\backslash B^n(0,R)}e^{\eta_\rho}\eta_\rho dy\right)\leq 0.\]
\[\]
Then

\begin{align*}
\mathop {\lim }\limits_{\rho \to C_n }J_\rho(u_\rho)&\geq
 \frac{1}{n C_n}\mathop {\lim }\limits_{\rho \to C_n }\int_{B^n(0,\varepsilon^{-1}_\rho)}e^{\eta_\rho(y)}\eta_\rho(y)dy+
  \frac{n-1}{n}\mathop {\lim }\limits_{\rho \to C_n }\int_{B^n(0,\varepsilon^{-1}_\rho)\backslash B^n(0,R)}e^{\eta_\rho}\eta_\rho dy\\
  &\ \ +(n-1) \ln \big( \frac{R^{\frac{n}{n-1}}}{1+R^{\frac{n}{n-1}}}\big)+\frac{n-1}{n}\ln \beta_n- \ln C_n.
\end{align*}
\[\]
One can prove this claim by contradiction. If not, there exists some $R_0>0$ such that
\[\mathop {\lim }\limits_{\rho \to C_n }\left( c_\rho+\frac{n^2}{n-1}\ln \varepsilon_\rho+n \ln \big( \frac{R_0^{\frac{n}{n-1}}}{1+R_0^{\frac{n}{n-1}}}\big)+\int_{B^n(0,\varepsilon^{-1}_\rho)\backslash B^n(0,R_0)}e^{\eta_\rho}\eta_\rho dy\right)\geq 0.\]
\[\]
 Consider the following inequality

\begin{equation}
\int_{B^n(0,\delta) \backslash  B^n(0,\varepsilon_\rho R_0 )} | \nabla  u_\rho|^ndx
    \geq
    \mathop {\mathop {\inf}\limits_{u|_{\partial B^n(0,\delta)} =u_\rho|_{\partial B^n(0,\delta)}} }
    \limits_{u|_{\partial B^n(0,\varepsilon_\rho R_0)}=u_\rho|_{\partial B^n(0,\varepsilon_\rho R_0)}}
   \int_{B^n(0,\delta)\backslash B^n(0,\varepsilon_\rho R_0)}|\nabla u|^ndx.
\label{inequatily3}
\end{equation}
\[\]
Then the left-hand side of inequality \eqref{inequatily3} can be written as

\begin{align*}
\int_{B^n(0,\delta) \backslash  B^n(0,\varepsilon_\rho R_0)} | \nabla  u_\rho|^n dx
 &=\left(\int_{B^n}-\int_{B^{n}\backslash B^n(0,\delta)}-\int_{B^n(0,\varepsilon_\rho R_0)}\right)| \nabla  u_\rho |^n dx
  \\
  &=:{\rm I}-{\rm II}-{\rm III}.
\end{align*}
\[\]
 For $\rm{I}$, using equality \eqref{eq8}, one can easily check that

\begin{align*}
    \mathop {\lim }\limits_{\rho \to C_n }{\rm I}&=( c_\rho-\ln\beta_n)
   \int_{B^n(0,\varepsilon^{-1}_\rho)}  e^{\eta_\rho(y)}dy
   +\int_{B^n(0,\varepsilon^{-1}_\rho)}  e^{\eta_\rho(y)}\eta_\rho(y)dy
   \\&=C_n( c_\rho-\ln\beta_n) + \mathop {\lim }\limits_{\rho \to C_n }\int_{B^n(0,\varepsilon^{-1}_\rho)}  e^{\eta_\rho(y)}\eta_\rho(y)dy.
\end{align*}
\[\]
For $\rm{II}$, recalling Corollary \ref{cor1} in the case of $m=1$, we have shown that as $\rho\rightarrow C_n$,
$ u_\rho\rightarrow u_0(x)$ in $C_{loc}^1(B^n\backslash \{0\})$,
where $u_0(x)$ satisfies the equation

\begin{equation*}\begin{cases}
&- \Delta_n u_0 = \sum\limits_{i = 1}^m \left(\frac{n}{n-1} \alpha_n\right)^{n-1}\delta_0, \ \ x\in B^n, \\
& u_0=0,\ \ x\in \partial B^n.
\end{cases}\end{equation*}
\[\]
By the relationship between the Green function of n-Laplacian operator with the singularity at $0$ and the Dirac function $\delta_0$,
one can immediately deduce that

\begin{align*}
    \mathop {\lim }\limits_{\rho \to C_n }{\rm II}& =\mathop {\lim }\limits_{\rho \to C_n }\int_{B^{n}\backslash B^n(0,\delta)}| \nabla  u_\rho(x) |^n dx
    =\left(\frac{n}{n-1}\alpha_n\right)^n \int_{B^{n}\backslash B^n(0,\delta)}| \nabla  G(x,0) |^n dx
     \\&
     =-C_n\frac{n^2}{n-1}\ln \delta.
\end{align*}
\[\]
Using the definition of $\eta_\rho$, we easily get that

\begin{align*}
    \mathop {\lim }\limits_{\rho \to C_n }{\rm III}&=\int_{B^n(0,\varepsilon_\rho R_0)}| \nabla  u_\rho(x) |^n dx
    = \mathop {\lim }\limits_{\rho \to C_n }\int_{B^n(0,R_0)}| \nabla  \eta_\rho (x)|^n dx
    \\&= \mathop {\lim }\limits_{\rho \to C_n }\int_{B^n(0,R_0)}  e^{\eta_\rho(y)}\eta_\rho(y)dy-C_n \ln\frac{\beta_n}{(1+R_0^{\frac{n}{n-1}})^n}.
\end{align*}
\[\]
As for the right-hand side, supposing

\[u=a\ln |x|+b .\]
By the following equation

\[
\left\{ {\begin{array}{*{20}{c}}
   u|_{\partial B^n(0,\delta)} =u_\rho|_{\partial B^n(0,\delta),} \\
   u|_{\partial B^n(0,\varepsilon_\rho R_0)}=u_\rho|_{\partial B^n(0,\varepsilon_\rho R_0),}  \\
\end{array}} \right.
\]
\[\]
and the definition of $u_\rho$ can yield that

\[
\left\{ {\begin{array}{*{20}{c}}
  a\ln \delta +b=u_\rho|_{\partial B^n(0,\delta),} \\
   a\ln \varepsilon_\rho R_0+b= u_\rho|_{\partial B^n(0,\varepsilon_\rho R_0)}.  \\
\end{array}} \right.
\]
\[\]
Hence, one can compute directly

\begin{equation}
a=\frac{-\frac{n^2}{n-1}\ln \delta-c_\rho-\ln\frac{1}{\left(1+R_0^{\frac{n}{n-1}}\right)^n}}{\ln \delta-\ln \varepsilon_\rho R_0}.
\label{a}
\end{equation}
\[\]
Thus,

\begin{align*}
\mathop {\mathop {\inf}\limits_{u|_{\partial B^n(0,\delta)} =u_\rho|_{\partial B^n(0,\delta)}} }
    \limits_{u|_{\partial B^n(0,\varepsilon_\rho R_0)}=u_\rho|_{\partial B^n(0,\varepsilon_\rho R_0)}}
    \|\nabla  u\|_{L^n(B^n(0,\delta) \backslash  B^n(0,\varepsilon_\rho R_0))}^n
 &=|a|^n\int_{B^n(0,\delta) \backslash  B^n(0,\varepsilon_\rho R_0)} \frac{1}{|x-x_\rho|^n}dx
  \\
  &=|a|^n \omega_{n-1}\left(\ln \delta-\ln \varepsilon_\rho R_0\right).
\end{align*}
\[\]
Then as $\rho\rightarrow C_n$, combining with (\ref{inequatily3}), (\ref{a}) and the results of ${\rm I,\ II,\ III}$, one can obtain the following inequality

\begin{equation}
\label{ineq4}
\begin{split}
    C_n c_\rho +C_n\frac{n^2}{n-1}\ln \delta-n C_n \ln\left(1+R_0^{\frac{n}{n-1}}\right)&+\mathop {\lim }\limits_{\rho \to C_n }\int_{B^n(0,\varepsilon_\rho^{-1})\backslash B^n(0,R_0)}  e^{\eta_\rho(y)}\eta_\rho(y)dy
    \\& \geq
    \omega_{n-1}\mathop {\lim }\limits_{\rho \to C_n }\frac{\left|-\frac{n^2}{n-1}\ln \delta-c_\rho-\ln\frac{1}{(1+R_0^{\frac{n}{n-1}})^n}\right|^n}
    {\left(\ln \delta-\ln \varepsilon_\rho R_0\right)^{n-1}}.
\end{split}
\end{equation}
\[\]
Obviously, $\mathop {\lim }\limits_{\rho \to C_n }\int_{B^n(0,\varepsilon_\rho^{-1})\backslash B^n(0,R_0)}  e^{\eta_\rho(y)}\eta_\rho(y)dy\leq 0$.
Hence, we conclude that

\begin{align*}
   & c_\rho +\frac{n^2}{n-1}\ln \delta-n \ln \left(1+R_0^{\frac{n}{n-1}}\right)   \leq \frac{n^2}{n-1}(\ln \delta-\ln \varepsilon_\rho R_0).
   \end{align*}
\[\]
Consequently, we can write  inequality (\ref{ineq4}) as

\[  \mathop {\lim }\limits_{\rho \to C_n }\left(c_\rho+\frac{n^2}{n-1}\ln \varepsilon_\rho+n \ln \big( \frac{R_0^{\frac{n}{n-1}}}{1+R_0^{\frac{n}{n-1}}}\big)+\int_{B^n(0,\varepsilon_\rho^{-1})\backslash B^n(0,R_0)}  e^{\eta_\rho(y)}\eta_\rho(y)dy  \right)\leq 0,
\]
\[\]
which contradicts with previous assumption,

$$c_\rho+\frac{n^2}{n-1}\ln \varepsilon_\rho+n \ln \big( \frac{R_0^{\frac{n}{n-1}}}{1+R_0^{\frac{n}{n-1}}}\big)+\int_{B^n(0,\varepsilon_\rho^{-1})\backslash B^n(0,R_0)}  e^{\eta_\rho(y)}\eta_\rho(y)dy>0.$$ \[\]
Thus, we accomplish the proof of the claim.
\medskip

By estimate \eqref{3-0}, $|\eta_\rho|\leq (\frac{\rho}{\omega_{n-1}})^{n-1}\ln |x|$ in $B^n(0,\varepsilon^{-1}_\rho)\backslash B^n(0,R)$.
Hence, it is easy to check that

$$\mathop {\lim }\limits_{\rho \to C_n }\int_{B^n(0,\varepsilon^{-1}_\rho)\backslash B^n(0,R)}e^{\eta_\rho}\eta_\rho dy= 0.$$
\[\]
To sum up,

\begin{align*}
\mathop {\lim }\limits_{R \to +\infty }\mathop {\lim }\limits_{\rho \to C_n }J_{C_n}(u_\rho)&\geq
\frac{1}{n C_n}\int_{\mathbb{R}^n}e^{\eta_0(y)}\eta_0(y)dy
 +\frac{n-1}{n}\ln \beta_n- \ln C_n.
\end{align*}

\end{proof}
\medskip

\emph{\bf Part 2}:
In this part, one can modify the standard solution to deduce an upper bound for $J_{C_n}$. Since the previous description about $\eta_0$, we easily obtain that $\tilde{\eta}_L(x):=\eta_0(\frac{x}{L})-n\ln L$ satisfies the equation

\begin{equation}
   { - \Delta_n \tilde{\eta}_L  =e^{\tilde{\eta}_L}\ \ {\rm in}\ \ B^n}.
\label{testfcn1}
\end{equation}
\[\]
We construct a test function sequence $\Phi_L:=\tilde{\eta}_L-\tilde{\eta}_L|_{\partial B^n}$. It is easy to check that $\Phi_L$ satisfies

\begin{equation}
\left\{ {\begin{array}{*{20}{c}}
   { - \Delta_n  \Phi_L  =e^{\tilde{\eta}_L}} & {{\rm in} \ B^n ,}  \\
   {\Phi_L = 0} & {\ \ {\rm on} \ \partial B^n }.  \\
\end{array}} \right.
\label{testfcn2}
\end{equation}
\[\]

Simple computations give that,

\begin{align*}
\mathop {\lim }\limits_{L \to 0}
J_{C_n}(\Phi_L)&=\mathop {\lim }\limits_{L \to 0 }\frac{1}{ n C_n}\int_{B^n} | \nabla  \Phi_L |^n dx- \ln \int_{B^n} e^ {\Phi_L}dx
  \\
  &=\mathop {\lim }\limits_{L \to 0}\frac{1}{ n C_n}\int_{B^n} e^{\tilde{\eta}_L}(\tilde{\eta}_L-\tilde{\eta}_L|_{\partial B^n}) dx- \ln \int_{B^n} e^ {\Phi_L}dx
  \\
  &=:\frac{1}{n C_n}(I_{11}-I_{12})-\ln I_2.
\end{align*}
For $I_{11}$, by the expression of $\tilde{\eta}_L$, we derive that

\begin{align*}
    I_{11}&=\mathop {\lim }\limits_{L \to 0 }\int_{B^n} e^{\tilde{\eta}_L}\tilde{\eta}_L dx
    =\mathop {\lim }\limits_{L \to 0 } \int_{B^n} e^{\eta_0(\frac{x}{L})-n \ln L}(\eta_0(\frac{x}{L})-n \ln L)dx
    \\&=\int_{\mathbb{R}^{n}}e^{\eta_0(y)}\eta_0(y)dy- nC_n\ln L.
\end{align*}
\[\]
For $I_{12}$, using the expression of $\tilde{\eta}_L$ again,
we have

\[I_{12}=\mathop {\lim }\limits_{L \to 0 }\int_{B^n} e^{\tilde{\eta}_L}\tilde{\eta}_L|_{\partial B^n} dx=\mathop {\lim }\limits_{L \to 0 }C_n\ln\frac{\beta_n}{L^n(1+|L|^{-\frac{n}{n-1}})^n}.\]
\[\]
Likewise, for $I_2$, we directly calculate

\[I_2=\mathop {\lim }\limits_{L \to 0 }\int_{B^n} e^ {\Phi_L}dx=\mathop {\lim }\limits_{L \to 0 }C_n \frac{L^n(1+|L|^{-\frac{n}{n-1}})^n}{\beta_n}.\]
\[\]
Combining the estimate $I_{11}$, $I_{12}$ and $I_2$, we conclude that

\begin{align*}
   \mathop {\lim }\limits_{L \to +\infty } J_{C_n}(\Phi_L)&=\mathop {\lim }\limits_{L \to 0}
   \frac{1}{ n C_n}\left(
 \int_{\mathbb{R}^{n}}e^{\eta_0(y)}\eta_0(y)dy-nC_n \ln L-C_n\ln\frac{\beta_n}{L^n(1+|L|^{-\frac{n}{n-1}})^n}\right)
    \\
    &\ \ \ -\ln C_n-\ln\frac{L^n(1+|L|^{-\frac{n}{n-1}})^n}{\beta_n}
    \\
    &=\frac{1}{nC_n}\int_{\mathbb{R}^{n}}e^{\eta_0(y)}\eta_0(y)dy+\frac{n-1}{n}\ln \beta_n- \ln C_n-\mathop {\lim }\limits_{L \to 0}\ln (1+|L|^{\frac{n}{n-1}})^{n-1}
     \\
     &=\frac{1}{nC_n}\int_{\mathbb{R}^{n}}e^{\eta_0(y)}\eta_0(y)dy+\frac{n-1}{n}\ln \beta_n- \ln C_n.
\end{align*}
\medskip

\section{The Proofs of Theorem \ref{thm3} and \ref{thm4}}

In this section, we shall establish the accurate lower bound of optimal concentration for the Moser-Onofri inequality on a general domain and give the criterion for the existence of extremals of the Moser-Onofri inequality. Since our methods are based on the $n$-harmonic transplantation, for reader's convenience, we also need to introduce some basic concepts and properties for $n$-capacity, Robin function and $n$-harmonic radius.
\medskip

\begin{definition}(Chapter 2 in \cite{Flucher}).
\label{def4.2}
The $n$-capacity of a set $A\subseteq \Omega$ with respect to $\Omega$ is defined as

\begin{equation}
    \label{eq4.1}
    n {\rm cap_\Omega}(A):=\inf\left\{\int_\Omega |\nabla u|^ndx: u\in W_0^{1,n}(\Omega), u\geq 1 \ {\rm on}\ A\right\}.
\end{equation}
\end{definition}
\[\]
We call $n{\rm mod}_\Omega(A):=n {\rm cap}_\Omega^{\frac{1}{1-n}}(A)$ the $n$-modulus of $A$ with respect to $\Omega$. A function which realizes the infimum \eqref{eq4.1} is called a $n$-capacity potential. The $n$-capacity potential satisfies equation

\begin{align*}
    -{\rm div} (|\nabla u|^{n-2}\nabla u) &=0\ \ {\rm in}\ \ \Omega\backslash A,
    \\
     u &=0\ \ {\rm in}\ \ \partial\Omega,
     \\
      u &=1\ \ {\rm in}\ \ \bar{A}.
\end{align*}
\[\]
Integration by parts leads to the boundary integral representation

\begin{equation}
    \label{4.3}
    n {\rm cap}_\Omega(A)= \int_{\partial A} |\nabla u|^{n-1}d\sigma.
\end{equation}
\medskip

\begin{definition}
\label{4-2def}
The Green function of $n$-Laplacian operator with the singularity at $x_0$ on the bounded domain is defined as the singular solution of Dirichlet problem

\begin{equation}\begin{cases}
-\Delta_n G_{x_0}(y)=\delta_{x_0}(y),\ \ y\in \Omega,\\
\ \ \ \ \ \  G_{x_0}(y)=0,\ \ \ \ \ \ \ y\in \partial\Omega.
\end{cases}\end{equation}
\[\]
The Green function of $n$-Laplacian operator can be decomposed into singular part and a regular part:

$$G_{x_0}(y)=K(|y-x_0|)-H_{x_0, \Omega}(y),\ \ K(|y-x_0|)=-\frac{n}{\alpha_n}\log(|y-x_0|).$$
\[\]
The regular part of the Green function of $n$-Laplacian operator on the bounded domain $\Omega$ evaluated at singularity $x_0$:

$$\tau_{\Omega}(x_0)=H_{x_0,\Omega}(x_0)$$
\[\]
is called the $n$-Robin function of $\Omega$ at $x_0$.
The $n$-harmonic radius $\rho_{\Omega}(x_0)$ at $x_0$ is defined by the relation

$$-\frac{n}{\alpha_n}\log (\rho_\Omega(x_0))=H_{x_0, \Omega}(x_0).$$
\end{definition}
\medskip

\begin{definition}
Define by $G_0$ the Green function of $n$-Laplacian operator on $B^n(0,r)$ with the singularity at $0$. For every positive radial function $U=\Phi\circ G_0(y): B^n(0,r)\rightarrow \mathbb{R}^{+}$ and $x_0\in \Omega$, we associate $u: \Phi\circ G_{x_0}(y):\Omega\rightarrow \mathbb{R}^{+}$. This transformation is called $n$-harmonic transplantation from $(B^n(0,r),0)$  to $(\Omega, x_0)$.
\end{definition}
\medskip

\begin{proposition}\label{pro4.4}
The $n$-harmonic transplantation has the following properties:
\medskip

(1) It preserves the $n$-Dirichlet-energy,

$$\int_{\Omega}|\nabla u|^ndx=\int_{B^n(0,r)}|\nabla U|^ndx.$$
\[\]
(2) If $r=\rho_{\Omega}(x_0)$, then

$$\int_{\Omega}F(u)dx\geq \int_{B^n(0,r)}F(U)dx=\rho_{\Omega}^n(x_0) \int_{B^n}F(U)dx.$$
\[\]
(3) If $F(U_k) \rightharpoonup c_0 \delta_0$ in the sense of measure, then $F(u_k)\rightharpoonup c_0 \delta_{x_0}$ in the sense of measure.
\end{proposition}
\medskip

\begin{proposition}
\label{prop4.4}(Theorem 9.5 of Chapter 9 in \cite{Flucher})
If the sets $(A_\varepsilon)$ concentrate at a point $x_0\in \Omega\cap \widetilde{\Omega}$ in the sense $A_\varepsilon\subseteq B^n(x_0, r_\varepsilon)$ with $r_\varepsilon\rightarrow 0$, then

 \[n{\rm mod}_\Omega(A_\varepsilon)=n{\rm mod}_{\widetilde{\Omega}}(A_\varepsilon)+\tau_{\widetilde{\Omega}}(x_0)-\tau_\Omega(x_0)+o(1)\]
 \[\]
 as $\varepsilon\rightarrow 0$.
\end{proposition}
\medskip

Now, we are in the position to give the accurate lower bound of optimal concentration for Moser-Onofri inequality on a general domain, namely, we shall provide the proof of Theorem \ref{thm3}. We first claim a basic fact that can be inferred from the proof of Theorem \ref{Th1.2th}:

$$F_{B^n}^{loc}(0)=\inf_{u\in W^{1,n}_0(B^n)}\Big(\frac{1}{n C_n}\int_{B^n} {|\nabla u|^n}dx - \ln \int_{B^n}{e^u} dx\Big).$$
\[\]
Indeed, since

$$\inf_{u\in W^{1,n}_0(B^n)}\Big(\frac{1}{n C_n}\int_{B^n} {|\nabla u|^n}dx - \ln \int_{B^n}{e^u} dx\Big)$$
\[\]
can not be achieved, if we define $w_k$ as the extremal function of

$$\inf_{u\in W^{1,n}_0(B^n)}\Big(\frac{1}{n(C_n-\epsilon_k)}\int_{B^n} {|\nabla u|^n}dx - \ln \int_{B^n}{e^u} dx\Big)$$
\[\]
with the $\epsilon_k\rightarrow 0$, then from the proof of Theorem \ref{Th1.2th}, we see that

$$\int_{B^n}e^{w_k}dx\rightarrow +\infty,\ \ \frac{e^{w_k}dx}{\int_{B^n}e^{w_k}dx}\rightharpoonup \delta_{0}$$
\[\]
and
\begin{equation}\begin{split}
\inf_{u\in W^{1,n}_0(B^n)}\Big(\frac{1}{nC_n}\int_{B^n} {|\nabla u|^n}dx - \ln \int_{B^n}{e^u} dx\Big)
&\ \ =\lim\limits_{k\rightarrow +\infty}\Big(\frac{1}{n C_n}\int_{B^n} {|\nabla w_k|^n}dx - \ln \int_{B^n}{e^{w_k}} dx\Big)\\
&\ \ \geq \frac{1}{nC_n}\int_{\mathbb{R}^{n}}e^{\eta_0(y)}\eta_0(y)dy+\frac{n-1}{n}\ln \beta_n- \ln C_n.
\end{split}\end{equation}
\[\]
Recall the \emph{Part 2} of the proof of Theorem \ref{Th1.2th}, we construct the suitable test function sequences $\Phi_{L}$ satisfying

$$\int_{B^n}e^{\Phi_L}dx\rightarrow +\infty,\ \ \frac{e^{\Phi_L}dx}{\int_{B^n}e^{\Phi_L}dx}\rightharpoonup \delta_{0}$$
\[\]
such that

$$\lim\limits_{k\rightarrow +\infty}\Big(\frac{1}{n C_n}\int_{B^n} {|\nabla \Phi_L|^n}dx - \ln \int_{B^n}{e^{\Phi_{L}}} dx\Big)=
\frac{1}{nC_n}\int_{\mathbb{R}^{n}}e^{\eta_0(y)}\eta_0(y)dy+\frac{n-1}{n}\ln \beta_n- \ln C_n.$$
\[\]
Combining the above estimate, we derive that
\begin{equation}\begin{split}\label{ad1}
 \inf_{u\in W^{1,n}_0(B^n)}\Big(\frac{1}{nC_n}\int_{B^n} {|\nabla u|^n}dx - \ln \int_{B^n}{e^u} dx\Big)=F_{B^n}^{loc}(0).
 \end{split}\end{equation}
A simple change of variable: $x\rightarrow x_0+Rx$ will directly yield
\begin{equation}\label{ad2}
F_{B^n(x_0,R)}^{loc}(x_0)=R^nF_{B^n}^{loc}(0).
\end{equation}
\medskip

\medskip

Now we start the proof of Theorem \ref{thm3}. Assume that $u_k\in W^{1,n}_0(\Omega)$ satisfies

$$\lim\limits_{k\rightarrow +\infty}\int_{\Omega}e^{u_k}dx=+\infty,\ \frac{e^{u_k}dx}{\int_{\Omega}e^{u_k}dx}\rightharpoonup \delta_{x_0}.$$
\[\]
Through Proposition \ref{pro4.4}, we see that

 $$\int_{\Omega}|\nabla u_k|^ndx=\int_{\Omega}|\nabla U_k|^2dx,\ \ \int_{\Omega}e^{u_k}dx\geq\rho^n_{\Omega}(x_0)\int_{B^n}e^{U_k}dx,\ \ \frac{e^{U_k}dx}{\int_{\Omega}e^{U_k}dx}\rightharpoonup \delta_{0}.$$
\[\]
Then we deduce that

{\small\begin{equation}\begin{split}
&\inf \left\{\lim\limits_{k\rightarrow +\infty}\Big(\frac{1}{n C_n}\int_{\Omega} {|\nabla u_k|^n}dx - \ln \int_{\Omega}{e^{u_k}} dx\Big)\ |\ \lim\limits_{k\rightarrow +\infty}\int_{\Omega}e^{u_k}dx=+\infty,\ \frac{e^{u_k}dx}{\int_{\Omega}e^{u_k}dx}\rightharpoonup \delta_{x_0}\right\}\\
&\ \ \leq \inf \left\{\lim\limits_{k\rightarrow +\infty}\Big(\frac{1}{n C_n}\int_{B^n} {|\nabla U_k|^n}dx - \ln \int_{B^n}{e^{U_k}} dx\Big)-n\ln \rho_{\Omega}(x_0)\ |\ \lim\limits_{k\rightarrow +\infty}\int_{B^n}e^{U_k}dx=+\infty,\ \frac{e^{U_k}dx}{\int_{B^n}e^{U_k}dx}\rightharpoonup \delta_{0}\right\}\\
&\ \ =\inf_{u\in W^{1,n}_0(B^n)}\Big(\frac{1}{n C_n}\int_{B^n} {|\nabla u|^n} - \ln \int_{B^n}{e^u} dx\Big)-n\ln \rho_{\Omega}(x_0).
\end{split}\end{equation}}
\[\]
Thus, in order to obtain our desired result, we just need to prove

\begin{equation}\begin{split}
&\lim\limits_{k\rightarrow +\infty}\Big(\frac{1}{n C_n}\int_{\Omega} {|\nabla u_k|^n}dx - \ln \int_{\Omega}{e^{u_k}} dx\Big)\\
&\ \ \geq \inf_{u\in W^{1,n}_0(B^n)}\Big(\frac{1}{n C_n}\int_{B^n} {|\nabla u|^n} - \ln \int_{B^n}{e^u} dx\Big)-n\ln \rho_{\Omega}(x_0).
\end{split}\end{equation}
\[\]
Since $\lim\limits_{k\rightarrow +\infty}\int_{\Omega}e^{u_k}dx=+\infty$, one can easily check that

\begin{equation}\begin{split}
&\lim\limits_{k\rightarrow +\infty}\Big(\frac{1}{n C_n}\int_{\Omega} {|\nabla u_k|^n}dx - \ln \int_{\Omega}{e^{u_k}} dx\Big)\\
&\ \ =\lim\limits_{k\rightarrow +\infty}\Big(\frac{1}{n C_n}\int_{\Omega} {|\nabla u_k|^n}dx-\ln \big(\int_{\{|u_k|\leq 1\}}e^{u_k}dx+\int_{\{|u_k|\geq 1\}}e^{u_k}dx\big)\Big)\\
&\ \ =\lim\limits_{k\rightarrow +\infty}\Big(\frac{1}{n C_n}\int_{\Omega} {|\nabla u_k|^n}dx-\ln \int_{\{|u_k|\geq 1\}}e^{u_k}dx\Big)
\end{split}\end{equation}
where the last inequality holds since $\lim\limits_{k\rightarrow+\infty}\int_{\Omega}e^{u_k}dx=+\infty$ and
$\int_{\{u_k\leq 1\}}e^{u_k}dx$ is bounded. On the other hand, since $$\frac{e^{u_k}dx}{\int_{\Omega}e^{u_k}dx}\rightharpoonup \delta_{x_0},$$
\[\]
we find that there exists $r_k\rightarrow 0$ such that $A_{k}\triangleq\{u_k\geq 1\}$ is included in $B_{r_k}(x_0)$.
Then we can replace $u_k$ below  level $1$ with the $n$-capacity potential of $A_k$. The resulting function is denoted by $v_k$. We apply the change
of the domain formula (Proposition \ref{prop4.4}) with
$\widetilde{\Omega}= B^n(x_0, \rho_{\Omega}(x_0))$ such that $n{\rm  mod}_{\widetilde\Omega}(A_k) =
n{\rm  mod}_{ \Omega}(A_k) +o(1)$. By the logarithmic structure of the fundamental singularity, a change of order $o(1)$ in the radius of $\Omega$ leads to a change of the same order in  the $n$-modulus. Thus, we can achieve that

\[n{\rm mod}_{B^n(x_0,\rho_\Omega(x_0)+o(1))}(A_k)= n{\rm mod}_{\Omega}(A_k)\]
\[\]
by increasing the radius of $\widetilde \Omega$ by $o(1)$. Hence, we deduce that
$$\int_{\Omega}|\nabla v_k|^2dx=\int_{B^n(x_0,\rho_\Omega(x_0)+o(1))}|\nabla w_k|^2dx,$$
where $w_k=v_k$ in $A_k$ and $w_k$ is the capacity potential of $ncap_{B^n(x_0,\rho_\Omega(x_0)+o(1))}(A_k)$ in $B^n(x_0,\rho_\Omega(x_0)+o(1))\setminus A_k$. Then it follows that

\begin{align*}
&\lim\limits_{k\rightarrow +\infty}\Big(\frac{1}{n C_n}\int_{\Omega} {|\nabla u_k|^n}dx - \ln \int_{\Omega}{e^{u_k}} dx\Big)\\
&\ \ =\lim\limits_{k\rightarrow +\infty}\Big(\frac{1}{n C_n}\int_{\Omega} {|\nabla u_k|^n}dx - \ln \int_{A_{k}}{e^{u_k}} dx\Big)\\
&\ \ \geq \lim\limits_{k\rightarrow +\infty}\Big(\frac{1}{n C_n}\int_{\Omega} {|\nabla v_k|^n}dx - \ln \int_{A_k}{e^{v_k}} dx\Big)\\
& \ \ \geq \lim\limits_{k\rightarrow +\infty}\Big(\frac{1}{n C_n}\int_{B^n(x_0,\rho_\Omega(x_0)+o_k(1))} {|\nabla w_k|^n}dx - \ln \int_{B^n(x_0,\rho_\Omega(x_0)+o_k(1))}{e^{w_k}} dx\Big)\\
&\ \ \geq \lim\limits_{k\rightarrow +\infty} F^{loc}_{B^n(0,\rho_\Omega(x_0)+o_k(1))}(x_0)\\
&\ \ =\inf_{u\in W^{1,n}_0(B^n)}\Big(\frac{1}{n C_n}\int_{B^n} {|\nabla u|^n}dx - \ln \int_{B^n}{e^u} dx\Big)-n\ln \rho_{\Omega}(x_0),
\end{align*}
where the last equality holds through \eqref{ad1} and \eqref{ad2}. Then we accomplish the proof of Theorem \ref{thm3}.
\medskip

\medskip

Now, we give the proof of Theorem \ref{thm4}. Let $u_k$ denote the extremal of subcritical Moser-Onofri inequality on a general domain $\Omega$:

$$\inf_{u\in W^{1,n}_0(\Omega)}\Big(\frac{1}{n (C_n-\epsilon_k)}\int_{\Omega} {|\nabla u|^n} - \ln \int_{\Omega}{e^u} dx\Big)$$
\[\]
with $\epsilon_k\rightarrow 0$. Then it is not difficult to check that

$$\lim\limits_{k\rightarrow +\infty}\Big(\frac{1}{n C_n}\int_{\Omega} {|\nabla u_k|^n} - \ln \int_{\Omega}{e^{u_k}} dx\Big)=C(n,\Omega).$$
\[\]
If $u_k$ is unbounded in $L^{\infty}(\Omega)$, arguing as what we did in Theorem \ref{Th1.2th}, we can derive that

 $$\lim\limits_{k\rightarrow +\infty}\int_{\Omega}e^{u_k}dx=+\infty,\ \
\frac{e^{u_k}}{\int_{\Omega}e^{u_k}dx} \rightharpoonup \delta_{x_0}.$$
\[\]
According to the definition of $F_{\Omega}^{loc}(x_0)$, we immediately conclude that

$$\lim\limits_{k\rightarrow +\infty}\Big(\frac{1}{nC_n}\int_{\Omega} {|\nabla u_k|^n} - \ln \int_{\Omega}{e^{u_k}} dx\Big)\geq F_{\Omega}^{loc}(x_0).$$
\[\]
In view of Theorem \ref{thm3}, we derive that

\begin{equation*}\begin{split}
&\lim\limits_{k\rightarrow +\infty}\Big(\frac{1}{n C_n}\int_{\Omega} {|\nabla u_k|^n} - \ln \int_{\Omega}{e^{u_k}} dx\Big)\\
&\ \ \geq \inf_{u\in W^{1,n}_0(B^n)}\Big(\frac{1}{n C_n}\int_{B^n} {|\nabla u|^n} - \ln \int_{B^n}{e^u} dx\Big)-n\sup_{x_0\in \Omega}\ln \rho_{\Omega}(x_0),
\end{split}\end{equation*}
\[\]
which contradicts with the assumption

$$C(n,\Omega)<\inf_{u\in W^{1,n}_0(B^n)}\Big(\frac{1}{n C_n}\int_{B^n} {|\nabla u|^n} - \ln \int_{B^n}{e^u} dx\Big)-n\sup_{x_0\in \Omega}\ln \rho_{\Omega}(x_0).$$
\[\]
Hence $u_k$ is bounded in $L^{\infty}(\Omega)$, it follows from the regular estimate for quasilinear operator that there exists $u\in W^{1,n}_{0}(\Omega)$ such that $u_k\rightarrow u$ in $C^{1}(\Omega)$ and

$$C(n,\Omega)=\lim\limits_{k\rightarrow +\infty}\Big(\frac{1}{n C_n}\int_{\Omega} {|\nabla u_k|^n} - \ln \int_{\Omega}{e^{u_k}} dx\Big)=\frac{1}{n C_n}\int_{\Omega} {|\nabla u|^n} - \ln \int_{\Omega}{e^u} dx.$$
\[\]
Then the proof of Theorem \ref{thm4} is accomplished.
\vskip0.1cm

{\bf Acknowledgement:}
\medskip

The authors wish to thank the referee for many constructive comments, suggestions
which have improved the exposition of the paper.
\medskip

{\bf Conflict of interest:} On behalf of all authors, the corresponding author states that there is no conflict of
interest.
\medskip

{\bf Data availability:} Data sharing not applicable to this article as no datasets were generated
or analysed during the current study.

\end{document}